\documentclass[a4paper]{article}

\usepackage[utf8x]{inputenc}
\usepackage{amsmath, amsfonts}

\usepackage[a4paper,top=3cm,bottom=2cm,left=2.5cm,right=2.5cm,marginparwidth=1.5cm]{geometry}

\usepackage{amsmath} 
\usepackage{graphicx} 
\usepackage{pdfpages}
\usepackage[colorlinks=true, allcolors=blue]{hyperref} 
\usepackage{xcolor}
\usepackage{authblk}
\usepackage[shortlabels]{enumitem}
\usepackage{ulem}

\graphicspath{ {./figures/} }

\newtheorem{thm}{Theorem}[section]
\newtheorem{lemma}[thm]{Lemma}
\newtheorem{prop}[thm]{Proposition}

\newtheorem{rem}[thm]{Remark}

\newcommand{\Z}{\ensuremath{\mathbb{Z}}}
\newcommand{\R}{\ensuremath{\mathbb{R}}}
\def\C{{\mathbb{C}}}

\def\cC{{\cal{C}}}
\def\C{{\mathbb{C}}}

\def\Re{{\textrm{Re}}}
\def\Im{{\textrm{Im}}}

\def\cC{{\it{C}}}
\def\Id{\rm{Id}}
\def\r{{\rm{r}}}

\def\Id{\mathrm{Id}}

\newcommand{\X}[3]{#1_{#3,#2}}
\newcommand{\Xe}[3]{#1^*_{#3,#2}}

\numberwithin{equation}{section}
\newenvironment{pf}{\vspace{0.5em}\noindent\textbf{Proof.} }{\quad \hfill $\Box$ \\ \vspace{0.2em}\\}

\title{On the rate of convergence to steady state\\ in a linear chromatography model}
\author[1]{Joaqu\'{\i}n Menacho}
\author[2]{Marta Pellicer}
\author[3]{J. Sol\`a-Morales}
\affil[1]{IQS School of Engineering, Universitat Ramon Llull, Via Augusta, 390, 08017 Barcelona, Catalunya, Spain, joaquin.menacho@iqs.edu}
\affil[2]{Departament d'Inform\`atica, Matem\`atica Aplicada i Estad\'istica, Universitat de Girona, C. Maria Aur\`elia Capmany, Girona, Catalunya, Spain, marta.pellicer@udg.edu}
\affil[3]{Departament de Matem\`atiques, Universitat Polit\`ecnica de Catalunya, Av. Diagonal
647, 08028 Barcelona, Catalunya, Spain, jc.sola-morales@upc.edu}

\begin{document}

\maketitle
\begin{abstract}
We study the rate of convergence to the steady state in the True Moving Bed model of linear chromatography, as a function of the six parameters that appear in the model. The model is a system of eight linear partial differential equations of hyperbolic type, coupled through the equations themselves and also through boundary conditions. We prove that the rate of convergence is given by a dominant eigenvalue, whose existence we prove by means of the Krein-Rutman Theorem, and by comparison arguments. We show how to construct a (not at all simple) characteristic function, whose roots are the eigenvalues. We also study the asymptotic profile of the solutions for large times, although this part is not purely analytical, but a combination of analytical and numerical techniques. Beyond the theoretical results, these models also offer explicit quantitative information: we apply all our results to a Case Study, namely the separation of omeprazole enantiomers. Finally, we consider a simpler limit case, where all the calculations become explicit.
\end{abstract}


\section{Introduction and summary}

Chromatography is one of the common techniques used in industry to separate the components of a mixture, taking advantage of their different affinities for the surface of an adsorbent. If the mixture passes in the liquid phase through a column containing a bed of adsorbent particles, each component of the mixture will advance at a different speed due to greater or lesser adsorption. Chromatographic separation takes advantage of these different velocities to separate the components.

In industry, the Simulated Moving Bed (SMB) technology is very often used to perform separations in continuous processes. The device consists of a series of cyclically connected chromatographic columns through which the mixture dissolved in a fluid phase circulates. Four ports are connected at certain points between the columns, delimiting four zones. At two ports, called raffinate and extract ports, part of the circulating liquid is extracted, while at the other two ports, the mixture (feed port) and the pure solvent (desorbent port) are injected. At regular time intervals, the four ports are moved one column in the direction of the liquid stream through a valve system, as can be seen in Figure \ref{F1} (left). This movement of the ports is essentially equivalent to a displacement of the solid phase in the opposite direction, thus simulating the countercurrent movement of the chromatographic bed, as described in M. Schulte et al. \cite{Schmidt-Traub5}, for example.

Although not implemented in industrial practice, the ideal countercurrent separator model is useful for analysing the behaviour of SMBs, since there is a correspondence between both systems (see Z.Ma, N.-H.L. Wang \cite{Ma-Wang}, A. Susanto et al. \cite{Schmidt-Traub7}, and V. Grosfils et al. \cite{Grosfils}). This ideal model, usually called True Moving Bed (TMB), consists of the liquid phase circulation with the mixture along a circular conduit while the solid phase circulates in the opposite direction (see Figure \ref{F1}, right). Four ports —with the same functions as those of an SMB— are located at four fixed locations. The TMB has a steady state, while the SMB only reaches a periodic state after a sufficiently long time. This makes solving a model of the TMB much simpler than that of the SMB. Therefore, analysing the TMB model is useful for designing and optimizing the operating parameters of an SMB, as, for instance, in I. Arrieta et al. \cite{Arrieta}, W. Jeong et al. \cite{Jeong}, and H.-J. Kang et al. \cite{Kang}.

\begin{figure}[htpb!]
  \centering
  \includegraphics[scale=0.4]{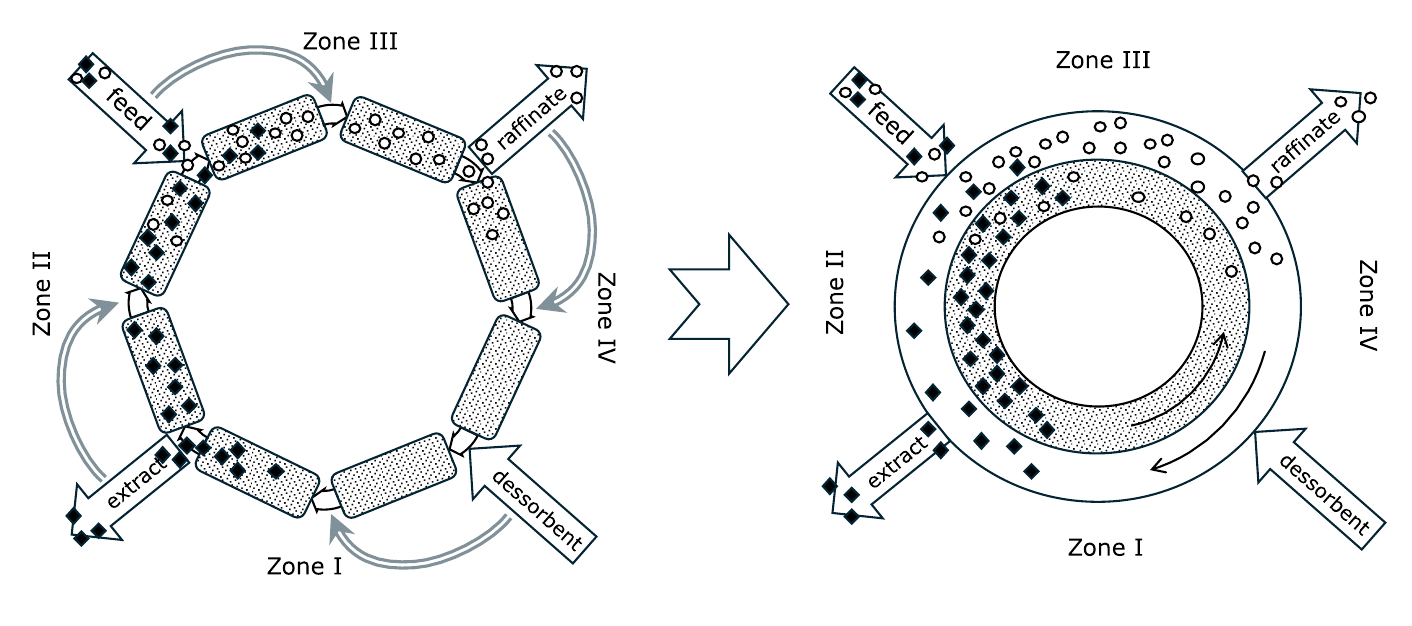}\\
  \caption{Left: scheme of an SMB. Right: scheme of a TMB model. The circles represent the less adsorbable component, while the squares represent the more adsorbable one. The arrows on the left figure represent the movement of the ports at regular time intervals, and in the right figure, the direction of circulation of the liquid and solid phases.}\label{F1}
\end{figure}

\subsection{The equations}
The simplest model of countercurrent linear chromatography is (see G. Guiochon et al. \cite{Guiochon2003}, R. Aris and N.R. Amundson \cite{ArisAmundson1973}, \cite{Grosfils}, or J.P. Aniceto and C.M. Silva \cite{Aniceto})
\begin{equation}\label{sysbasic}
\begin{cases}
    c_t+u c_x=D c_{xx}-Fk(Hc-q) \\
    q_t-u_s q_x=k(Hc-q)
\end{cases}
\end{equation}
where $H$ is the equilibrium constant, $F=(1-\epsilon)/\epsilon$ is the phase volume ratio ($\epsilon$ being the column void fraction), $u$ and $u_s$ are the fluid and solid phase velocities, $D$ is the diffusion coefficient and $k \hspace{1mm} (s^{-1})$ is a kinetic constant. Here we will consider the model without the diffusion term, that is, $D \approx 0$. The unknowns $c=c(x,t)$ and $q=q(x,t)$ are the concentrations of the solute in the liquid and solid phases, respectively.

For a four-zone TMB model, $L$ being the length of each zone, the model \eqref{sysbasic} stands for every zone, each one with a different liquid phase velocity $u_i, i=1, 2, 3, 4$ and with the following continuity boundary conditions:
\begin{equation}\label{bcbasic}
\begin{cases}
    q(x_i^{-})=q(x_i^{+}), x_i=-2L^{+}, 2L^{-}, \pm L, 0\\ c(x_i^{-})=c(x_i^{+}), x_i=\pm L \\ u_1 c(-2L^{+})=u_4c(2L^{-}) \\ f_{in}+u_2 c(0^{-})=u_3c(0^{+})
\end{cases}
\end{equation}
where $f_{in}=c_{feed} Q_{feed}$ is the solute mass flow injected with a concentration $c_{feed}$. The inequalities $u_1>u_2$ and $u_3>u_4$ will be adopted, and they mean that $x=\mp L$ are extraction ports. Also for the inequalities $u_2<u_3$ and $u_4<u_1$, that mean that $x=0$ and $x=\pm 2L$ are injection ports.

Let us now define dimensionless variables as
\begin{equation}\label{dimless}
x'=x / L,\quad c'=\sqrt{H/F}\frac{c}{c_{feed}},\quad q'=\frac{q}{c_{feed}},\quad t'=\frac{u_s}{L}t
\end{equation}
and the system \eqref{sysbasic} becomes (without the primes)
\begin{equation}\label{eqn0t}
\begin{cases}
    c_t+v_i c_x=R(-P^2 c+Pq) \\ q_t-q_x=R(Pc-q),
\end{cases}
\end{equation}
for $x\in I_i=(x_i, x_{i+1})$, $i=1,2,3,4$ and $x_i=-3+i$, where $v_i=u_i/u_s$, $R=kL/u_s$ and $P=\sqrt{FH}$, with the continuity conditions \eqref{bcbasic} that now become
\begin{equation}\label{bcn0t}
\begin{cases}
    q(x_i^{-})=q(x_i^{+}), x_i=-2^{+}, 2^{-}, \pm 1, 0\\ c(x_i^{-})=c(x_i^{+}), x_i=\pm 1 \\ v_1 c(-2^{+})=v_4 c(2^{-}) \\ f_0+v_2 c(0^{-})=v_3 c(0^{+})
\end{cases}
\end{equation}
where $f_0=\sqrt{H/F} \hspace{1mm}Q_{feed}/(u_sL^2)$ is the dimensionless inflow.

The above inequalities between the $u_i$'s now take the form
\begin{equation}\label{ineq}
v_1>v_2>0, \ v_2<v_3, \ v_3>v_4>0 \text{ and } v_4<v_1.
\end{equation}
These inequalities represent the fact that ports $1$ and $3$ ($x=-1$, $x=1$, respectively) are extraction ports, and ports $2$ and $4$ ($x=0$, and either $x=2$ or $x=-2$, respectively) are injection ones. Throughout the paper, there will be situations in which these inequalities can be equalities. That is, $v_1=v_2$ and $v_3=v_4$. Then, \eqref{ineq} would imply that $v_i=v>0$ for $i=1,2,3,4$. We will explicitly indicate when this special situation applies. Moreover, in some cases, we can consider these strict inequalities as non-strict ones: we will also specify this when it applies.

The situation described by \eqref{eqn0t}, \eqref{bcn0t}, and \eqref{ineq} is illustrated in the following figure.

\begin{figure}[htpb]
  \centering
\includegraphics[width=13.5cm]{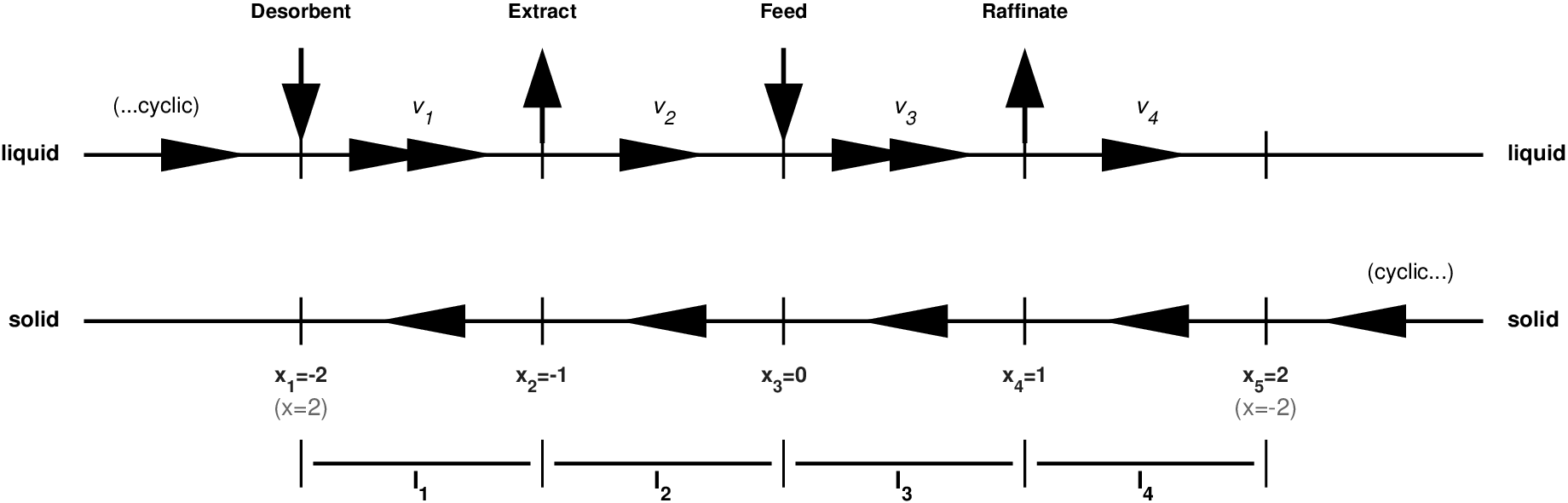}
\caption{Scheme of the situation described by \eqref{eqn0t}, \eqref{bcn0t}, and \eqref{ineq}.}\label{AbstractStright}
 \end{figure}

\begin{rem}
In the limit case of the inequalities \eqref{ineq}, namely when $v_i=v>0,\ i=1,2,3,4$, the contour conditions \eqref{bcn0t} become simply continuity conditions for $c$ and $q$ at all the boundaries of each interval $I_i$. This is a case of not much practical interest, but it is the unique case where we have been able to find an explicit and complete picture of the eigenvalues. It will be studied in Section \ref{limit}.
\end{rem}

The last problem \eqref{eqn0t} and \eqref{bcn0t}, with the inequalities \eqref{ineq}, is the dimensionless TMB model of a four-zone chromatography device. Due to its linearity, the perturbations of a steady-state obey the homogeneous case, that is, when there is no injection of solute ($f_0=0$). In this paper, we are mainly studying this homogeneous case. In the case of several different solutes, which is often the case of practical interest, if there is no interaction at the level of adsorption, due to linearity, it is enough to study the evolution of each substance independently.

In the previous work of J. Menacho and J. Sol\`a-Morales \cite{Menacho-SM}, some of the authors of this paper proved the exponential convergence to a unique steady state of this system. By means of asymptotic analysis, an approximation of the steady-state concentration profiles was found, for the cases of low mass transfer resistance -that is, $k$ (or $R$) large. In the present work, we are interested in the rate of convergence to the steady state and the long-time profile of the limiting concentrations.

This rate is directly related to the duration of the transient states of an SMB. Except for some notable works (C.B. Ching et al. \cite{Ching1991}, I.B.R. Nogueira et al.\cite{Nogueira2017}, and T.H. Oh et al. \cite{Oh2019}), the study of these transients has so far received not much attention in published research.

More recently, the authors of the present paper studied a simpler system very similar to \eqref{sysbasic} that arises mostly in biomathematics when considering a correlated random walk model in J. Menacho et al. \cite{Menacho-Pell-SM}. There, the model only considers one zone, with a unique velocity $v$  in both directions,  with homogeneous Dirichlet boundary conditions at the inflow points. The equality of the velocities in the two directions produces a dimensionless form of the problem with only three parameters ($F, K$, and $S=v/(KL)$) instead of $R, P$ and four $v_i$. For this correlated random walk case, the authors describe the spectrum of the operator, with relatively simple formulas to find the so-called dominant eigenvalue and its corresponding eigenfunction, and prove that they give the optimal rate of convergence to the equilibrium, and the asymptotic convergence profile, respectively.

It is worth writing here the equations for the eigenvalues $\lambda$ and eigenfunctions $(\tilde{c}(x), \tilde{q}(x))^T$ of problem \eqref{eqn0t}-\eqref{bcn0t}, with $f_0=0$. They correspond to solutions of \eqref{eqn0} of the form $c(t,x)=e^{\lambda t}\tilde{c}(x)$, $q(t,x)=e^{\lambda t}\tilde{q}(x)$. For simplicity of notation we write $c$ instead of $\tilde{c}(x)$ and $q$ instead of $\tilde{q}(x)$:
\begin{equation}\label{eqn0}
\left\{ \begin{array}{rll}
\lambda c+v_i c_x &= RP(-Pc+q)\quad &\hbox{for}\: x\in I_i, \; i=1,\dots 4, \\
\lambda q-q_x &= R(Pc-q)\quad &\hbox{for}\: x\in I_i, \; i=1,\dots 4,
\end{array}\right.
\end{equation}
with the same previous boundary conditions, in the homogeneous and time-independent case:
\begin{equation}\label{bcn0}
\left\{ \begin{array}{lll}
  &c(x_i^-)=c(x_i^+) & \hbox{for}\: x_i=-1,1,\\
  &v_1c(-2^+)=v_4c(2^-),&\\
  &v_3c(0^+)= v_2c(0^-),& \\
  &q(x_i^-)=q(x_i^+) & \hbox{for}\: x_i=-1,0,1,\pm 2 \hbox{. ($2^+$ means $-2^-$)}.
\end{array}\right.
\end{equation}
We look for non-trivial solutions of \eqref{eqn0}-\eqref{bcn0}, real or complex. But our main purpose is to find a real solution $(c_0(x)>0,q_0(x)>0)^T$, with $\lambda=\lambda_0<0$ (principal or dominant eigenvalue), that we prove that exists. We note that for the eigenfunctions, $c(x)$ will be discontinuous at $x=0, \pm 2$ and $q(x)$ will be continuous, but not necessarily of class $\mathcal{C}^1$.

\subsection{Summary of results}
In this section, we present a summary of the main results of the present work, together with a justification of the techniques that we have used to prove them.

Section \ref{sec:functional} is devoted to the functional framework of the evolution problem \eqref{eqn0t}-\eqref{bcn0t} with $f_0=0$. This problem can be written in the abstract form
\begin{equation}\label{abs}
(c_t,q_t)^T = A(c,q)^T ,\ t\in[0,+\infty)
\end{equation}
for a suitable unbounded operator on a Hilbert space $X$, of $L^2$ type, defined only on a certain domain $D(A)$ made of functions $(c,q)^T$ that are continuous except at the jumps of $c(x)$ at $x=0,\pm2$. Then, the eigenvalue problem \eqref{eqn0}-\eqref{bcn0} can be written in the abstract form as $A(c,q)^T=\lambda (c,q)^T$.

We understand as one of our main achievements the writing of the eigenvalue problem  \eqref{eqn0}-\eqref{bcn0} as a single scalar equation of the form $\Delta(\lambda)=0$ (see \eqref{eq:f}), which deserves to be called the characteristic equation, and where $\lambda$ can be real or complex. This characteristic equation is constructed through a Return or Monodromy map, that is, a linear $2\times 2$ matrix $C(\lambda)$, and the equation $\Delta(\lambda)=0$ is equivalent to asking the matrix $C(\lambda)$ to have $1$ as one of its eigenvalues. This will be presented in Section \ref{sec:carac}.

The complicated form of the characteristic equation and of the description of the eigenfunctions is one of the reasons for the use of functional analytic techniques, such as the theory of Positive Operators and the Krein-Rutman Theorem, even to prove that the equation $\Delta(\lambda)=0$ has at least one solution. Section \ref{posit} is devoted to giving the tools and proof of the main theorem of the present paper, that is, Theorem \ref{thfa} below.
\begin{thm}\label{thfa}\textbf{(Large time behaviour of the solutions)}
We consider the operator $A$ of the evolution problem \eqref{abs} (that is, the abstract form of problem \eqref{eqn0t}-\eqref{bcn0t} with $f_0=0$). Then, the following assertions are true.
\begin{enumerate}[(i)]
\item If the inequalities \eqref{ineq} are satisfied, there exists a (unique) largest real eigenvalue of $A$, $\lambda_0\in\mathbb{R}$, $\lambda_0<0$.
Moreover, if $\lambda$ is any eigenvalue of $A$, then $Re(\lambda)\leq \lambda_0$. This $\lambda_0$ has a strictly positive eigenfunction $(c_0(x),q_0(x))^T$, in the sense that $c_0(x)>0$ and $q_0(x)>0$ for all $x\in[-2,2]$ . It also satisfies the (non-optimal) bound $-M_0\le\lambda_0<0,$ where $M_0$ is the following positive quantity
\begin{equation}\label{eps0}
M_0: = R-\dfrac{ R^2P^2v^{min}}{v^{min} \left( \frac{v^{max}-v^{min}}{2}\right)+v^{max}\left(RP^2+1\right)} >0
\end{equation}
where $v_{max}:=\max\{v_1,v_2,v_3, v_4\}$ and $v_{min}:=\min\{v_1,v_2,v_3,v_4\}$.
\item If $(c(x,t),q(x,t))^T=e^{tA}(c(x,0),q(x,0))^T$ with a (a.e.) bounded initial condition $(c(x,0),q(x,0))^T$, then there exist a $K>0$, depending only on the initial condition, such that
\begin{equation}\label{final}
 -K e^{t\lambda_0}\left(|c_0(x)|,|q_0(x)|\right)^T \leq \left(c(x,t),q(x,t)\right)^T\le K e^{t\lambda_0}\left(|c_0(x)|,|q_0(x)|\right)^T
\end{equation}
for all $t\ge 0$ and a.e. in $x$. If $(c,q)^T$ is a non-real solution, then the same previous inequality holds for the real and imaginary parts of the solution. The number $K$ can be simply chosen as the smallest real number for which the inequalities \eqref{final} hold at $t=0$.

\item In the limit case $v_1=v_2$ and $v_3=v_4$, the largest real eigenvalue is $\lambda_0=0$. The corresponding solutions also satisfy inequality \eqref{final} (with $\lambda_0=0$).
\end{enumerate}
 \end{thm}
\begin{rem}
If $Re(\lambda)\leq \lambda_0$ for all $\lambda\in \sigma(A)$ (spectrum of the operator $A$), then we will call $\lambda_0$ to be the dominant eigenvalue of $A$.
\end{rem}
\begin{rem}
Point (ii) of Theorem \ref{thfa} says that $\lambda_0$ gives the optimal decay rate of solutions of problem \eqref{eqn0t}-\eqref{bcn0t} with $f_0=0$ and, hence, the optimal rate of convergence of solutions of problem \eqref{eqn0t}-\eqref{bcn0t} to its steady state. This completes the work of \cite{Menacho-SM}.
\end{rem}
\begin{rem}
We will see that if $(c(x,0),q(x,0))^T\in D(A)$ then the solutions are piecewise continuous in $x$ and $t$ (with the known exceptions for $c(x,t)$ at $x=0,\pm2$), and then the above a.e. caution conditions are not needed.
\end{rem}
\begin{rem}
Observe in \eqref{eps0} that $-M_0>-R$. It is perhaps not accurate, but it becomes easier to remember that $-R<\lambda_0<0$.
\end{rem}
We want to emphasize that the decay rate of the solutions of an abstract dissipative linear equation of the form $U_t=A\,U$ is given by the spectral bound of $A$ (the dominant eigenvalue of $A$ in our case), at least in finite dimensional systems, or in systems where the Spectral Mapping Theorem can be applied to $e^{tA}$ for $t>0$. As we said, in our previous work \cite{Menacho-Pell-SM}, we proved similar results to Theorem \ref{thfa} for a simpler model. In that case, the evolution operator $e^{tA}$ needed to be eventually compact, as this enabled us to apply the Spectral Mapping Theorem. With that, we were able to prove that the dominant eigenvalue gave the optimal decay rate of the solutions. This also enabled us to prove that the corresponding eigenfunction gave the profile of the solutions for large times.

However, in the present work, we have not been able to prove that $e^{tA}$ is compact or eventually compact, and it is possible that this property does not hold (see Remark \ref{rem:compact}). This is another of the reasons why we rely on Functional Analysis techniques: instead of using the Spectral Mapping Theorem, we overcome this difficulty by applying the Krein--Rutman Theorem together with positivity arguments. This is done in Section \ref{posit}, where the complete proof of Theorem \ref{thfa} is provided.

The results on the optimal decay rate of solutions of \eqref{eqn0t}-\eqref{bcn0t} established in Theorem \ref{thfa} suggest that the large time profile of the decay of these solutions is governed by the eigenfunction associated with the dominant eigenvalue $\lambda_0$.
We have not been able to prove this result analytically, but all our numerical experiments consistently support this behaviour (see Fig. \ref{fig:conv}). We therefore formulate it as a reasonable conjecture and leave its proof as a direction for future work (see also Fig. \ref{fig:conv}):

\textit{Let $(c(x,t),q(x,t))^T=e^{tA}(c(x,0),q(x,0))^T$ be a solution of the initial value problem. Then
\begin{equation}\label{conj}
(c(x,t),q(x,t))^T= M_1
e^{t\lambda_0}(c_0(x),q_0(x))^T+o(e^{t\lambda_0})
\end{equation}
when $t\to\infty$, where $M_1=\langle (c(x,0),q(x,0))^T, (c^*_0(x),q^*_0(x))^T\rangle_X$, $\langle\ ,\ \rangle_X$ is the scalar product in $X$ and
$(c^*_0(x),q^*_0(x))^T$ is the eigenfunction associated to $\lambda_0$ of the adjoint operator $A^*$, normalized (for instance) in such a way that
$\langle (c_0(x),q_0(x))^T, (c^*_0(x),q^*_0(x))^T\rangle_X=1.$
}

Section \ref{sens} is devoted to sensitivity analysis of $\lambda_0$, that is, to analyse the variation of the dominant eigenvalue $\lambda_0$ with respect to the six parameters of the problem. That is, $v_i$ for $i=1,\ldots 4$, $R$ and $P$. As we will see, this is done using the adjoint problem and the derivative problem, which will be computed for each parameter.

As we have said, neither the calculation of $C(\lambda)$ nor the solution of the equation $\Delta(\lambda)=0$ is very straightforward, and has to be done by using a mixture of analytical and numerical methods. This has an exception, namely the limit conservative case $v_1=v_2=v_3=v_4=v>0$, where the description of the whole spectrum can be done quite explicitly. This situation represents no extraction nor injection through the ports, and perhaps is not very important in practice, but it will be a relevant theoretical example. This will be done in Section \ref{limit}.

In Section \ref{case} we present an example of the complete program for a set of realistic values of the parameters, that correspond to a case of omeprazole enantiomers described in F. Wei et al. \cite{Wei}. The full program consists of the following: first, computing the steady state and a general solution of \eqref{eqn0t}-\eqref{bcn0t} with $f_0=1$, illustrating its convergence towards the steady state (see Fig. \ref{steady}, \ref{ivp} and \ref{fig:conv}); second, deriving and analysing the characteristic function $\Delta(\lambda)$, determining the dominant eigenvalue $\lambda_0$, and obtaining $(c_0(x), q_0(x))$, $(c_0^*(x), q_0^*(x))$ (the eigenfunctions of the problem and the adjoint problem corresponding to $\lambda_0$, respectively); and, third, performing the sensitivity analysis of $\lambda_0$ with respect to small perturbations of the six parameters.

Finally, the paper concludes with three appendices, where we provide the explicit and complete calculations and procedures required throughout the work: the construction of four matrices $M_i$, $i=1,2,3,4$, used in the computation of $\Delta(\lambda)$, the computation of the eigenfunctions coefficients, the sensitivity analysis, the asymptotic behaviour of $\Delta(\lambda)$ as $|\lambda|\to\infty$, and the numerical method employed to simulate the solution of the evolution problem \eqref{eqn0t}--\eqref{bcn0t}. All these details are included to describe the (non-trivial) implementation of the computations. For clarity of exposition, we present them in separate appendices.

\section{Abstract Setting}\label{sec:functional}
We write the evolution problem \eqref{eqn0t} with boundary conditions \eqref{bcn0} in the abstract form given in \eqref{abs},
where $A:D(A)\subset X\rightarrow X$ is the linear operator
\begin{equation} \label{opA}
\left.\left(A\begin{pmatrix}c\\q\end{pmatrix}\right)\right|_{I_i}=A_i \left(\left.\begin{pmatrix}c\\q\end{pmatrix}\right|_{I_i}\right):=\begin{pmatrix}
-v_i\frac{d}{dx}-R P^2 \,\Id& RP\, \Id\\RP\, \Id&\frac{d}{dx}-R\, \Id
\end{pmatrix}\begin{pmatrix}c\\q\end{pmatrix}\ \textrm{for }  I_i=[x_i,x_{i+1}],\ i = 1,2,3,4.
\end{equation}
Here, $X=L^2(-2,2)\times L^2(-2,2),$
real or complex,
with its natural scalar product
$$\langle (c_1,q_1)^T,(c_2,q_2)^T\rangle_X =\int_{-2}^2 \left(c_1(x)\overline{c_2(x)} +q_1(x)\overline{q_2(x)}\right)\, dx,$$
and
the domain $D(A)$ is made of the Sobolev class functions
\begin{equation*}
D(A)=\left\{(c, q)^T \in \prod _{i=1}^4 H^1(I_i)\times\prod _{i=1}^4H^1(I_i), \text{ satisfying the conditions \eqref{bcn0} } \right\}.
\end{equation*}
\begin{rem}
Observe that the domain condition for $q$ is as merely asking $q\in H^1(-2,2)$, with $q(-2)=q(2)$.
\end{rem}
We define now the following natural quadratic energy for solutions $(c(x,t),q(x,t))^T$ of \eqref{abs}:
\begin{equation}\label{eq:energy}
E(t) =\dfrac{1}{2}\int_{-2}^{2} (|c(x,t)|^2+|q(x,t)|^2)\ dx.
\end{equation}
We write $|c|^2$ and $|q|^2$ in the energy instead of $c^2$ and $q^2$ in order to also deal with possibly complex-valued solutions.
\begin{lemma}\label{lem:dEdt}
If inequalities \eqref{ineq} hold (even if they are not strict), then
\begin{enumerate}[(i)]
\item
For any $(c,q)^T\in D(A)$, the operator $A$ defined in \eqref{opA} is dissipative, that is,
\begin{equation}\label{eq:maxmon}
{\rm Re}(\langle A(c,q)^T, (c,q)^T\rangle_X\le 0
\end{equation}
\item
For any $(c(t),q(t))^T$ solution of \eqref{abs}, the energy \eqref{eq:energy} is decreasing in $t$, that is,
$$\dfrac{dE}{dt} \leq 0,\ \textrm{ for all }t\ge 0.$$
\end{enumerate}
\end{lemma}
\begin{pf}
The first part of the lemma is obtained from the next calculations, using both equations of system \eqref{eqn0t}, integrating, and using the boundary conditions \eqref{bcn0} for $q$:
\begin{align*}
{\Re}\left(\langle A(c,q)^T,(c,q)^T\rangle_X\right)
=\frac{1}{2}\sum_{i=1}^4 v_i\left(-|c(x_{i+1})|^2+|c(x_i)|^2\right) -\int_{-2}^2 R|Pc-q|^2\ dx
\end{align*}
Using again \eqref{bcn0}, but now for $c$, we can write the previous expression as
\begin{multline}\label{dEdt}
{\rm Re}(\langle A(c,q)^T, (c,q)^T\rangle_X=
-R\int_{-2}^{2}|Pc-q|^2 dx
-(v_1-v_2)|c(-1)|^2-(v_3-v_4)|c(1)|^2-\\ -(v_1-v_4)|c(-2^+)||c(2^-)|-(v_3-v_2)|c(0^-)||c(0^+)|\le 0
\end{multline}
which is non-negative because of the inequalities \eqref{ineq} (even if they were not strict).

The second part of the lemma follows directly from the first one. Since $E(t)=\frac{1}{2}\langle (c(t),q(t))^T, (c(t),q(t))^T\rangle_X$ we have
$\frac{d}{dt}E(t)=\frac{1}{2}\langle A(c,q)^T, (c,q)^T\rangle_X+\frac{1}{2}\langle (c,q)^T, A(c,q)^T\rangle_X={\rm Re}(\langle A(c,q)^T, (c,q)^T\rangle_X)$, which gives the desired result.
\end{pf}
\begin{rem}\label{rem:dissipE}
In the inequality \eqref{dEdt}, we can see the two phenomena that are responsible for the dissipation of the energy: the equilibrium between the two phases (integral term), and the exit or dilution of the species through the ports (boundary terms).
\end{rem}
\begin{prop}
If the inequalities \eqref{ineq} hold (even if they are not strict), the operator $A$ defined in \eqref{abs} generates a $C_0$-semigroup of contractions $T(t)=e^{At}$ on the space $X$.
\end{prop}
\begin{pf}
We first start by noticing that $D(A)$ is dense in $X$, which is easy to see. Second, we can see that the range $\mathcal{R}(s\Id-A)=X$ for some $s>0$. This second property is more delicate and will be proved later on, in Lemma \ref{invers}. Indeed, we will prove there that $s\Id-A$ has a continuous inverse defined in $X$ for all $s>0$, and that $\|(s\Id-A)^{-1}\|\le s^{-1}$. Finally, we take into account the dissipativeness inequality \eqref{eq:maxmon}.

With these three ingredients, by the well-known Lumer-Phillips Theorem (see, for instance, Theorem 4.3 of A. Pazy \cite{Pazy}), we conclude that the operator $A$ generates a $C_0$-semigroup $e^{tA}$ of contractions.
\end{pf}
As a consequence of the previous proposition, we have the following existence and uniqueness result for the initial value problem:
\begin{thm}
For any initial condition $(c_0,q_0)^T\in D(A)$, there exists a unique solution $(c(t),q(t))^T=e^{tA}(c_0,q_0)^T$ of \eqref{abs} satisfying that $(c(t),q(t))^T\in \cC^0([0,\infty),D(A))\cap\cC^1([0,\infty)),X)$ (\textsl{strict solution}). Also, for any $(c_0,q_0)\in X$, there also exists a unique $(c(t),q(t))^T\in\cC^0([0,\infty)),X)$ which satisfies \eqref{abs} in a weak sense (\textsl{mild solution}).
\end{thm}
\begin{rem}
Mild solutions are not as irrelevant in practice as they may seem at first sight, since they admit discontinuous initial conditions.
\end{rem}
Lemma \ref{lem:dEdt} has the following implications.
\begin{prop}\label{prop:Reneg}
Let $\lambda$ be any eigenvalue of $A$ defined in \eqref{opA}. If inequalities \eqref{ineq} hold (even if they are not strict), then $\Re(\lambda)\leq 0$. Moreover, if at least one of the four dissipative inequalities \eqref{ineq} holds in the strict sense, then $\Re(\lambda)<0$.
\end{prop}
\begin{pf}
Let $(c,q)^T$ be the eigenfunction corresponding to $\lambda$. We have that $\lambda\langle (c,q)^T, (c,q)^T\rangle= \langle A (c,q)^T,(c,q)^T\rangle$, and by the inequality \eqref{eq:maxmon}
we conclude that $\Re(\lambda)\leq 0$.

Let us now assume that at least one of the four dissipative inequalities \eqref{ineq} holds in the strict sense. For instance, assume $v_1>v_2$ (the other cases are analogous). Suppose that $\Re(\lambda)=0$ (with $(c,q)\not\equiv (0,0)$) and we will arrive to a contradiction. As ${\rm Re}\langle A(c,q)^T,(c,q)^T\rangle=\Re(\lambda)\,\|(c,q)\|^2$, and going back to \eqref{dEdt}, we see that
$$
0=-\int_{-2}^2R|Pc-q|^2\ dx-(v_1-v_2)|c(-1)|^2-(v_3-v_4)|c(1)|^2-(v_1-v_4)|c(-2^+)||c(2^-)|-(v_3-v_2)|c(0^-)||c(0^+)|,
$$
and consequently $Pc\equiv q$ and $c(-1)=0$, and then also $q(-1)=0$. Then the system \eqref{eqn0}, that is, $A(c,q)^T=\lambda(c,q)^T$, reads on each interval $I_i$ as the homogeneous and uncoupled linear system of ordinary differential equations
\begin{equation*}
\left\{ \begin{array}{ll}
&\dfrac{dc}{dx}=-\dfrac{\lambda}{v_i} c_n  \\[10pt]
&\dfrac{dq}{dx} =\lambda q
\end{array}\right.
\end{equation*}
with the boundary conditions \eqref{bcn0}. The continuity and jump boundary conditions \eqref{bcn0} imply that if the final value in an interval is $(0,0)^T$, then also $(0,0)^T$ has to be the initial condition in the next interval. And then $c(-1)=q(-1)=0$ implies that $c(x)\equiv q(x)\equiv 0$ on $[-1,0]$ and $c(x)\equiv q(x)\equiv 0$ on $[-2,2]$, which is a contradiction.
\end{pf}
\section{The Characteristic Equation}\label{sec:carac}
In this section, we derive a characteristic equation for the eigenvalue problem \eqref{eqn0}-\eqref{bcn0}. Its solutions are the eigenvalues $\lambda$ of this problem, with $\lambda_0$ being the dominant one. As we will see, the obtention and manipulation of this characteristic equation is mathematically involved, both analytically and numerically. For this reason, we regard the results presented in this section as one of the main contributions of the present work.

For each $\lambda\in \C$ let us write \eqref{eqn0} as an explicit system of ordinary differential equations on each interval $I_i$, $i=1,2,3,4$:
\begin{equation}\label{eqn1v}
\dfrac{d}{dx}
\begin{pmatrix}
c\\q
\end{pmatrix}
=
\begin{pmatrix}
-\dfrac{\lambda+P^2R}{v_i}&\dfrac{RP}{v_i}\\-RP&\lambda+R
\end{pmatrix}
\begin{pmatrix}
c\\q
\end{pmatrix}=:F_i(\lambda)
\begin{pmatrix}
c\\q
\end{pmatrix},\text { for } x\in[x_i,x_{i+1}]
\end{equation}
whose fundamental matrix solution is $M_i(\lambda, x):=e^{(x-x_i) F_i(\lambda)}$,  for $x\in[x_i,x_{i+1}]$. Then, the solution of \eqref{eqn0} for a given $\lambda$ and satisfying the jump conditions \eqref{bcn0} with initial condition $(c(-1^+),q(-1^+))^T$ at $x=-1^+$ can be obtained by the following Transfer Matrix process:
$$\begin{pmatrix}
c(x)\\q(x)
\end{pmatrix}=M_2(\lambda,x)\begin{pmatrix}
c(-1^+)\\q(-1^+)
\end{pmatrix}, \text{ for }x\in[-1^+,0^-]$$
$$\begin{pmatrix}
c(x)\\q(x)
\end{pmatrix}=M_3(\lambda,x)\cdot\begin{pmatrix} v_2/v_3&0\\0&1\end{pmatrix}\cdot M_2(\lambda, 0^-)\begin{pmatrix}
c(-1^+)\\q(-1^+)
\end{pmatrix}, \text{ for }x\in[0^+,1^-]$$
$$\begin{pmatrix}
c(x)\\q(x)
\end{pmatrix}=M_4(\lambda,x)\cdot M_3(\lambda,1^-)\cdot\begin{pmatrix} v_2/v_3&0\\0&1\end{pmatrix}\cdot M_2(\lambda,0^-)\begin{pmatrix}
c(-1^+)\\q(-1^+)
\end{pmatrix}, \text{ for }x\in[1^+,2^-]$$
and
\begin{align*}
\begin{pmatrix}
c(x)\\q(x)
\end{pmatrix}=M_1(\lambda,x)\cdot\begin{pmatrix} v_4/v_1&0\\0&1\end{pmatrix}\cdot M_4(\lambda, 2^-)\begin{pmatrix} v_2/v_3&0\\0&1\end{pmatrix} \cdot M_3(\lambda,1^-)\cdot\begin{pmatrix} v_2/v_3&0\\0&1\end{pmatrix}\cdot M_2(\lambda,0^-)\begin{pmatrix}
c(-1^+)\\q(-1^+)
\end{pmatrix},& \\
\textrm{ for } x\in[-2^+,-1^-].&
\end{align*}
(recall that $2^-=-2^+$). Now we need the extra condition that $(c(-1^-),q(-1^-))^T$ results equal to the initial condition $(c(-1^+),q(-1^+))^T$. This extra condition is achieved only when $\lambda$ is an eigenvalue of the operator $A$. This amounts to saying that the Return Map, that is, the following $2\times 2$ matrix:
\begin{equation}\label{eq:C}
C(\lambda):= M_1(\lambda,-1^-)\cdot
\begin{pmatrix} v_4/v_1&0\\0&1\end{pmatrix}\cdot M_4 (\lambda, 2^-)\cdot M_3(\lambda, 1^-)\cdot
\begin{pmatrix} v_2/v_3&0\\0&1\end{pmatrix}\cdot M_2 (\lambda,0^-),
\end{equation}
has $\mu=1$ as one of its two eigenvalues.
Or, in other words, that the characteristic equation we are looking for is $$\Delta(\lambda)=0,$$
where the function $\Delta(\lambda)$ is defined by
\begin{equation}\label{eq:f}
\Delta(\lambda):=-\det(C(\lambda)-{\rm{Id}}), \text{ or }\Delta(\lambda):={\rm trace}(C(\lambda))-\det(C(\lambda))-1.
\end{equation}

Observe that one can avoid the intermediate values of $x_i$ and also write
\begin{equation}\label{simpleC}
C(\lambda)= e^{F_1(\lambda)}\cdot
\begin{pmatrix} v_4/v_1&0\\0&1\end{pmatrix}\cdot e^{F_4(\lambda)}\cdot e^{F_3(\lambda)}\cdot
\begin{pmatrix} v_2/v_3&0\\0&1\end{pmatrix}\cdot e^{F_2(\lambda)},
\end{equation}
or working with cyclic permutations of these six matrices, since cyclic permutations do not change the trace of their products. This would be equivalent to starting the cycle at a different $x_i$.

In contrast with the trace of $C(\lambda)$, the determinant of $C(\lambda)$ is easy to compute, using Liouville's formula for $\det(e^{F_i(\lambda)})$. We obtain

\begin{equation}\label{detC}
\det (C(\lambda))=  \dfrac{v_2v_4}{v_1v_3}\exp\left({2\left(\frac{\alpha_1(\lambda)}{v_1} + \frac{\alpha_2(\lambda)}{v_2} + \frac{\alpha_3(\lambda)}{v_3} + \frac{\alpha_4(\lambda)}{v_4}\right)}\right),
\end{equation}
where
\begin{equation}\label{alpha}
\alpha_i(\lambda) := \frac{(v_i - 1)\lambda + (v_i - P^2)R}{2}.
\end{equation}
The previous formula can also be written as
\begin{equation}\label{eq:detC2}
\det (C(\lambda)) = \dfrac{v_2v_4}{v_1v_3}\exp\left( \lambda \left(4-\sum_{i=1}^4\dfrac{1}{v_i} \right) + R\left( 4-P^2\sum_{i=1}^4\dfrac{1}{v_i}\right)   \right),
\end{equation}
where the dependence on $\lambda$ is more explicit. Observe that $\frac{v_2v_4}{v_1v_3}<1$ because of \eqref{ineq}.

Both to compute the function $\Delta(\lambda)$ or to obtain the whole eigenfunction for a given eigenvalue, it is useful to have an analytic expression for the matrices $M_i(\lambda, x)$. For this purpose, we use the eigenvalues and eigenvectors of $F_i(\lambda)$ (see \eqref{eqn1v}).

The eigenvalues $\X{\nu}{j}{i}(\lambda)$ of $F_i(\lambda)$ ($i=1,2,3,4$, $j=1,2$) are
\begin{equation}\label{eq:nu}
 \X{\nu}{1}{i}(\lambda) = \frac{\alpha_i(\lambda)}{v_i} + \sqrt{\left(\frac{\alpha_i(\lambda)}{v_i}\right)^2 + \frac{\beta(\lambda)}{v_i}} ,\qquad
 \X{\nu}{2}{i}(\lambda) = \frac{\alpha_i(\lambda)}{v_i} - \sqrt{\left(\frac{\alpha_i(\lambda)}{v_i}\right)^2 + \frac{\beta(\lambda)}{v_i}}
\end{equation}
with $\alpha_i(\lambda)$ given by \eqref{alpha} and
\begin{equation}\label{eq:beta}
\beta(\lambda) = \lambda^2 + \lambda R (1 + P^2).
\end{equation}
The corresponding eigenvectors $(\X{c}{j}{i}(\lambda),\X{q}{j}{i}(\lambda))$ ($i=1,2,3,4$, $j=1,2$) of $F_i(\lambda)$ are
\begin{equation*}
 \begin{pmatrix} \X{c}{j}{i}(\lambda) \\ \X{q}{j}{i}(\lambda) \end{pmatrix} = \begin{pmatrix} \lambda + R - \X{\nu}{j}{i}(\lambda) \\ RP \end{pmatrix}.
\end{equation*}
Then, the general solution of \eqref{eqn1v} (when $\X{\nu}{1}{i}\ne \X{\nu}{2}{i})$ is:
\begin{equation}\label{eq:fupA}
\begin{pmatrix} c_i(\lambda,x) \\ q_i(\lambda,x) \end{pmatrix} = \X{C}{1}{i} \begin{pmatrix} \lambda + R - \X{\nu}{1}{i}(\lambda) \\ RP \end{pmatrix} e^{\X{\nu}{1}{i}(\lambda)x} + \X{C}{2}{i} \begin{pmatrix} \lambda + R - \X{\nu}{2}{i}(\lambda) \\ RP \end{pmatrix} e^{\X{\nu}{2}{i}(\lambda)x},\textrm{ for }i = 1, 2, 3, 4.
\end{equation}
The matrices $M_i(\lambda,x)=e^{(x-x_i)F_i }$, for $x\in\R$, and also the matrix $C(\lambda)$ can be explicitly obtained after these expressions, as it can be seen in Appendix \ref{sec:Mi}. The eigenfunction coefficients can be obtained imposing the boundary conditions \eqref{bcn0} in each interval $I_i$: the constants  $\X{C}{j}{i}$ are then obtained by solving an $8\times 8$ linear system (see Appendix \ref{sec:Cji} for details).

\subsection{Asymptotics of $\Delta(\lambda)$}\label{sec:delta}
As an indication for the numerical treatment of the evaluations of the function $C(\lambda)$, we present a result on the asymptotic values of $\Delta(\lambda)$ for $\lambda\in\R$ and $|\lambda|$ large. The conclusion is that $\Delta(\lambda)\to +\infty$ as $\lambda\to\pm\infty$, $\lambda\in \R$. We present here the main steps for obtaining this result, but the complete computations are given in Appendix \ref{sec:appendix2}.
\begin{rem}\label{rem:Ogran}
As usual, for a real function $\varphi$ of a real variable $\lambda$ we write $\varphi(\lambda)=O(1)$ if $|\varphi(\lambda)|\le M$ for $|\lambda|$ large enough. We say $\varphi(\lambda)=O^+(1)$ if $0<m\le\varphi(\lambda)\le M$ for $|\lambda|$ large enough. So, we know the sign of $O^+(1)$, but not the sign of $O(1)$.
\end{rem}
For $\lambda \to -\infty$, we have from formula \eqref{eq:traceinfinity} that
\begin{equation}\label{traceminusinfty}
{\rm trace}(C(\lambda)) = O^+(1) \exp\left(\left(\sum_{i=1}^4\dfrac{1}{v_i}\right)|\lambda|\right)
\end{equation}
and we recall formula \eqref{eq:detC2} for $\det(\lambda)$, which implies that
\begin{align*}
\det(C(\lambda)) &= O^+(1) \exp\left(  \left(\sum_{i=1}^4\dfrac{1}{v_i}-4 \right)|\lambda|    \right)
\end{align*}
Therefore, it is clear that ${\rm trace}(C(\lambda))$ dominates $\det(C(\lambda))$ as $\lambda\to -\infty$, and $\Delta(\lambda)\to +\infty$ when $\lambda\to -\infty$. More precisely,
$$\Delta(\lambda)\sim\exp\left(\left(\sum_{i=1}^4\dfrac{1}{v_i}\right)|\lambda|\right) \textrm{ as }\lambda\to -\infty.$$
For $\lambda \to+\infty$, formula \eqref{eq:traceinfinity} gives
\begin{equation}\label{traceplusinfty}
{\rm trace}(C(\lambda) = O^+(1) \exp({4\lambda})
\end{equation}
and formula \eqref{eq:detC2} implies that
\begin{align*}
\det(C(\lambda)) &=O^+(1)\exp\left(  \left(4-\sum_{i=1}^4\dfrac{1}{v_i} \right)\lambda    \right).
\end{align*}
Therefore, it is clear that ${\rm trace}(C(\lambda))$ also dominates $\det(C(\lambda))$ as $\lambda\to +\infty$. More precisely,
$$\Delta(\lambda) \sim \exp(4\lambda)\to +\infty\textrm{ as }\lambda\to +\infty.$$
As we said, the above asymptotic expressions are carefully derived in Appendix \ref{sec:appendix2} below. But a heuristic deduction would be the following. Using the form \eqref{eqn1v} of the matrices $F_i(\lambda)$ when $|\lambda|\to \pm\infty$, we can say that, essentially,
\[
\left\{\begin{aligned}
F_i(\lambda) &\sim \begin{pmatrix}  \dfrac{|\lambda|}{v_i} & 0\\ 0 & -|\lambda|\end{pmatrix},
 \quad  \exp(F_i(\lambda)) \sim  \begin{pmatrix}  \exp\left(\dfrac{|\lambda|}{v_i} \right)& 0\\ 0 & 0\end{pmatrix},
& \quad\text{ if } \lambda\to -\infty \\[1em]
F_i(\lambda) &\sim  \begin{pmatrix}  \dfrac{\lambda}{v_i} & 0\\ 0 & -\lambda\end{pmatrix},
 \qquad  \exp(F_i(\lambda)) \sim   \begin{pmatrix}  0& 0\\ 0 & \exp(\lambda)\end{pmatrix},
& \quad \text{ if } \lambda\to +\infty
\end{aligned}\right.
\]
Then, performing the five products of diagonal matrices in \eqref{simpleC} for $\lambda\to-\infty$ and for $\lambda\to +\infty$, one gets \eqref{traceminusinfty} and \eqref{traceplusinfty} as reasonable results. But this is just an educated guess. The reason why this is not completely rigorous is that it is an exponential of a matrix with a very large term on the diagonal, and this term somehow contaminates terms that have other positions in the matrix. So we need the Appendix \ref{sec:appendix2} for the rigorous proof.

In Fig. \ref{mult} (left), there are four plots of the graph of the function $\Delta(\lambda)$ near the dominant eigenvalue for realistic values of the parameters. Concretely, for $v_1=1.53, v_2=1.12,  v_3=1.43, v_4=1.02$, $P=1.03$, and four different values of $R$: $R=9$ (dashed-blue), $R=18$ (continuous-black), $R=27$ (dotted-red), and $R=36$ (dash-dot green). The case $R=18$ corresponds to a real industrial case, which will be studied in more detail in Section \ref{case}.  Only the graph for $R=18$ is in the right scale. To fit in the same scale, the other three have been expanded or compressed: for $R=9$, the plot is that of $330\cdot\Delta$, for $R=27$ is that of $\Delta/100$, and for $R=36$ is that of $\Delta/10^4$. In Fig. \ref{mult} (right), we can see a general view of the same cases, now for $(2/\pi)\arctan(\Delta(\lambda))$ (and all at the same scale).

\begin{figure}[htpb]
  \centering
  \includegraphics[scale=0.5]{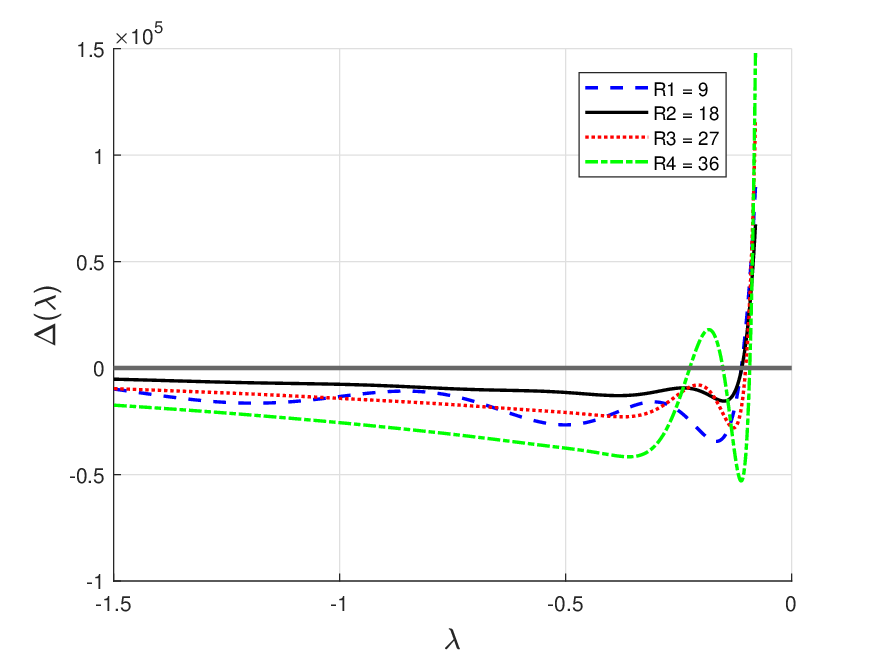} \includegraphics[scale=0.5]{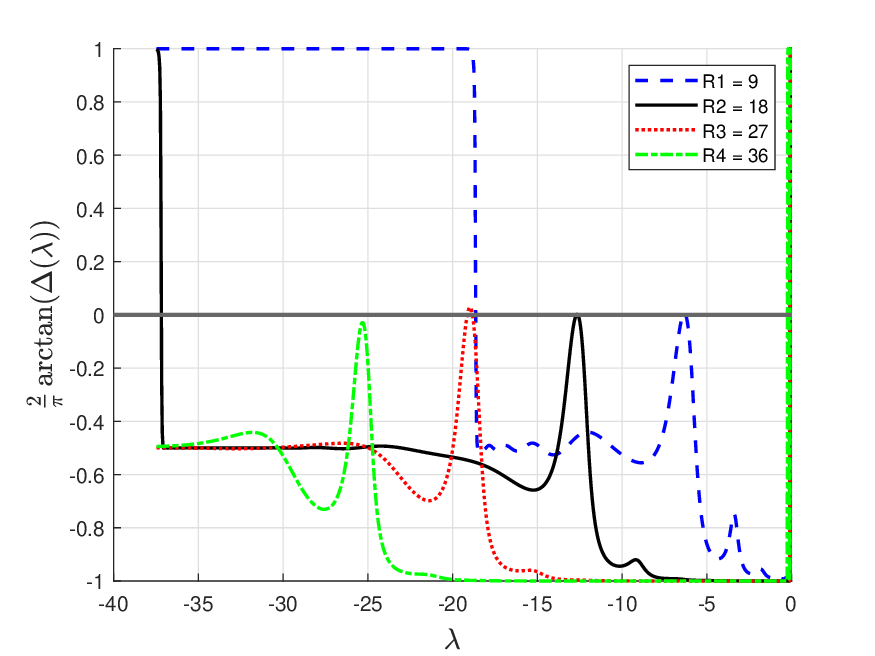}
  \caption{Left: partial plots of the graph of the function $\Delta(\lambda)$, for $R=9, 18, 27, 36$ ($v_i$ for $i=1,\ldots 4$, and $P$ as in the case study of Section \ref{case}), at quite different scales. Right: complete plots of the graph of the function $(2/\pi)\arctan(\Delta(\lambda))$, for the same values of $R$, at the same scale.}\label{mult}
 \end{figure}

Observe that the graphs are very steep near their largest real zero (see Fig. \ref{mult} (left)). Figure \ref{mult} (right) also shows the rapid growth of $\Delta(\lambda)$ when $|\lambda|\to\infty$, leading to numerical instabilities when dealing with it. Moreover, since we know that $\Delta(\lambda)\to\infty$ when $\lambda\to-\infty$, there should be at least a second real zero for sufficiently negative values of $\lambda$ (maybe more than two, depending on the values of the parameters), as Fig. \ref{mult} (right) clearly shows.
The numerical experiences further indicate that the graph remains relatively flat near $\Delta=-1$ between the first and last real zero, with some small fluctuations that may produce new real zeroes (see Fig. \ref{mult}, right), and it is very steep outside this interval, giving it a U-shaped profile. Although these features are numerically evident, we have not been able to establish them analytically. Such information would have allowed the direct application of Bolzano's Theorem, for instance, to guarantee the existence of an eigenvalue (one of the hypotheses required to apply Krein--Rutman's Theorem). Since we have not succeeded in this straightforward approach, we have instead relied on the four Functional Analysis lemmas in Section \ref{posit} to prove Theorem \ref{thfa}.


%
%

%
\section{Positivity and consequences: existence of $\lambda_0$}\label{posit}
In this section, we make a complete proof of the three parts of Theorem \ref{thfa}. We recall that one of the main purposes of this work is to prove the existence of a real dominant eigenvalue $\lambda_0$ and to prove that it gives the rate of convergence to the equilibrium of the solutions of problem \eqref{eqn0t}-\eqref{bcn0t}, as well as to localize it as accurately as possible. As we have seen in Section \ref{sec:carac}, this is equivalent to finding the largest real root of the equation  $\Delta(\lambda)=0$, where $\Delta(\lambda)$ is given in \eqref{eq:f}. But, as we can see in Section \ref{sec:carac} (and also Appendix \ref{sec:appendix2}), the analytical (and numerical) study of this function is highly intricate. The reason for the present Functional Analysis section is due to this fact: these are the tools that we have found to prove the main results of this paper, given in Theorem \ref{thfa}.

Let $(A,D(A))$ and $X$ be the operator and spaces defined in Section \ref{sec:functional}. We consider in $X$ the closed convex cone $X_+$ of (almost everywhere) non-negative components $c(x)\ge 0, q(x)\ge 0$. It induces a partial order relation in $X$. Since $X_++(-X_+)=X$, the cone $X_+$ is a reproducing cone.

The proofs are mainly based on the positivity of the operators $(-A+s\Id)^{-1}$ for $s>0$ and $e^{tA}$ for $t>0$ with respect to the cone $X_+$. Also on the application of the Krein-Rutman Theorem to the operators $(-A+s\Id)^{-1}$, whose compactness will be proved, though not to  $e^{tA}$, whose possible compactness is unclear (see Remark \ref{rem:compact}). The proofs will be done through the following series of four lemmas. These lemmas will be used to prove Theorem \ref{thfa}.
\begin{lemma}\label{invers}
For $s\in(0,+\infty)$ the linear operator $(-A+s \Id)^{-1}$ exists, and is a bounded and compact operator from $X$ to $X$ with $\|(-A+s \Id)^{-1}\|\le s^{-1}$.
\end{lemma}
\begin{pf}
Given $(\varphi_1,\varphi_2)^T\in X$ we have to find $(c,q)^T\in D(A)$ such that $-A(c,q)^T+s(c,q)^T=(\varphi_1,\varphi_2)^T$, that is
\begin{equation}\label{eqn1nh}
\dfrac{d}{dx}
\begin{pmatrix}
c\\q
\end{pmatrix}
=
\begin{pmatrix}
-\dfrac{s+P^2R}{v_i}&\dfrac{RP}{v_i}\\-RP&s+R
\end{pmatrix}
\begin{pmatrix}
c\\q
\end{pmatrix}+\begin{pmatrix}-\dfrac{1}{v_i}\varphi_1\\ \varphi_2\end{pmatrix} \hspace{0.25cm} \textrm{      in each interval $I_i=[x_i,x_{i+1}]$} .
\end{equation}
We start with the case when the functions $\varphi_1,\varphi_2$ are continuous on each closed interval $I_i$. In this case, this is a standard non-homogeneous linear system of ordinary differential equations on each interval $I_i$. As in Section \ref{sec:carac}, the simplest way to solve it with the boundary conditions \eqref{bcn0} is perhaps in the order $I_2,I_3,I_4, I_1$ (that is, starting at $x=-1^+$). Observe that the initial condition of the equation in the next interval is obtained after the final value in the previous interval (and imposing the boundary conditions \eqref{bcn0}). Then, at the end, one has to merely impose that $(c(-1^-),q(-1^-))^T$ coincides with $(c(-1^+),q(-1^+))^T$.

To show that such a solution exists, we can construct a solution of the form $(c,q)^T = (c^h,q^h)^T + (c^p,q^p)^T  $, where $(c^h,q^h)^T$ is a solution of the homogeneous version of system \eqref{eqn1nh}, and $(c^p,q^p)^T$ is a particular solution of system \eqref{eqn1nh}. In both cases, we impose that the boundary conditions \eqref{bcn0} must be fulfilled, but we still have to impose the conditions at the first end $x=-1^+$.

In the case of the particular solution, we impose $(c^p(-1^+),q^p(-1^+))^T=(0,0)^T$, for example. Then, the particular solution $(c^p,q^p)^T$ exists in each interval $I_i$ because of the Existence and Uniqueness Theorem. In the case of the solution of the homogeneous system, we impose $(c^h(-1^+),q^h(-1^+))^T=(c_1,q_1)^T$, where $(c_1,q_1)^T$ is the algebraic solution of the linear system
$$C(s)(c_1,q_1)^T+(c^p(-1^-),q^p(-1^-))^T=(c_1,q_1)^T$$
with $C(s)$ being the Return Map given by the $2\times2$ matrix defined in \eqref{eq:C}. Recall that applying $C(s)$ to any solution at $x=-1^+$ means to run a whole cycle, that is,  $C(s)(c_1,q_1)^T =C(s)(c^h(-1^+),q^h(-1^+))^T = (c^h(-1^-),q^h(-1^-))^T$. We can assure that the above linear system has an algebraic solution since we know that the matrix $C(s)$ has not the eigenvalue $1$ when $s>0$ (because, in this case, $\lambda=s>0$ would be an eigenvalue of $A$, which is not possible, see Proposition \ref{prop:Reneg}). Hence, again by the Existence and Uniqueness Theorem for ODEs, $(c^h,q^h)^T$ exists. Observe that, with the previous definitions, we have $(c(-1^+),q(-1^+))^T = (c(-1^-),q(-1^-))^T$, and $(c,q)^T$ is a solution of the system \eqref{eqn1nh} with boundary conditions \eqref{bcn0}, as desired.

Still with the functions $\varphi_1,\varphi_2$ being continuous on each $I_i$,  we multiply the equality $-A(c,q)^T+s(c,q)^T=(\varphi_1,\varphi_2)^T$ by $(c,q)^T$ with the scalar product $\langle\cdot,\cdot\rangle_X$, and we get that
$$\langle-A(c,q)^T,(c,q)^T\rangle_X+s\langle (c,q)^T,(c,q)^T\rangle_X=
\langle (\varphi_1,\varphi_2)^T,(c,q)^T\rangle_X.$$
Then, we can take the real part of the previous equality and recall that $\Re\left(\langle-A(c,q)^T,(c,q)^T\rangle_X\right)\ge 0$ (see \eqref{eq:maxmon}). Consequently, by using the Cauchy-Schwarz inequality, we obtain
$$
s\|(c,q)^T\|_X^2\le \Re\left(\langle (\varphi_1,\varphi_2)^T,(c,q)^T\rangle_X\right)\le\|(\varphi_1,\varphi_2)^T\|_X\|(c,q)^T\|_X.
$$
From this we get the useful estimate $\|(c,q)^T\|_X\le s^{-1}\|(\varphi_1,\varphi_2)^T\|_X$,
that means the continuity of $(-A+s\Id)^{-1}$ from $X$ to $X$ when restricted to continuous functions in each $I_i$.

Since these continuous functions are dense in $X$, the operator $(-A+s\Id)^{-1}$ can be extended to a continuous operator from $X\to X$. By reasoning along sequences $(\varphi_1^{k},\varphi_2^{k})$ which are convergent in $X$, we see that the extended operator $(-A+s\Id)^{-1}$ gives precisely the solutions of \eqref{eqn1nh} when the derivatives ${d}/{dx}$ are understood in the usual weak-$H^1$ sense. This in particular implies that $(-A+s\Id)^{-1}(X)\subset D(A)$. One also easily sees that  $ (-A+s \Id)^{-1}$ is bounded from $X$ to $D(A)$, when $D(A)$ is equipped with the piecewise $H^1$ norm. And then the compactness from $X$ to $X$ follows from the compactness of the embedding $D(A)\subset X$.
\end{pf}
\begin{lemma}\label{pos}
For each $s\in(0,+\infty)$, the operator $(-A+s{\rm{Id}})^{-1}$ is positive.
\end{lemma}
\begin{pf}
We are going to show that for any $s>0$ we have $(-A+s\Id)^{-1}X_+\subset X_+$ , that is, if $\varphi_1,\varphi_2\geq 0$ are given functions in $X$, and we have $(c,q)$ satisfying
\begin{equation*}
\left\{\begin{array}{rll}
v_i c_x+RP^2c-RPq+s c&=\,\varphi_1\ge 0 &\hbox{ (a.e.)}\\
-q_x-RPc+Rq+s q&=\, \varphi_2\ge 0 &\hbox{ (a.e.)}
\end{array}
\right.
\end{equation*}
and the boundary conditions \eqref{eqn0t}, then the solution $(c,q)\in D(A)$ is also non-negative.

Let us start again with the case where the functions $\varphi_1,\varphi_2$ are continuous on each closed interval $I_i$, and therefore the solutions $c(x),q(x)$ are of class $\cC^1$ on each interval. Suppose that the absolute minimum of the set $\{Pc(x),q(x)\,;\ -2\le x\le2\}\subset\R$  is $Pc(x_0)<0$ with $x_0$ being an interior point of some of the intervals, and we arrive at a contradiction:
\begin{equation}\label{contc}
\underbrace{v_i c_x(x_0)}_{=0}+RP\underbrace{(Pc(x_0)-q(x_0))}_{\le 0}+ \underbrace{s c(x_0)}_{<0}=\underbrace{\varphi_1(x_0)}_{\ge 0}.
\end{equation}
If this minimum is strictly negative but it is attained at $q(x_0)$ with $x_0$ being interior to some of the intervals, the situation would be similar, but now looking at the second equation:
\begin{equation}\label{contq}
\underbrace{-q_x(x_0)}_{=0}+R\underbrace{(q(x_0)-Pc(x_0))}_{\le 0}+ \underbrace{s q(x_0)}_{<0}=\underbrace{\varphi_2(x_0)}_{\ge 0}.
\end{equation}
And this is also a contradiction.

This strictly negative minimum can also be achieved at $q(x_0)$ with $x_0=-2,-1,0,1,2$ (points where $q(x)$ is continuous, although it may not be differentiable) or else at  $Pc(x_0)$ with $x_0=\pm 1$ (points where $c(x)$ is continuous, although it may not be differentiable). In these cases, we will necessarily have that $q_x(x_0^+)\ge 0$ or that $v_ic_x(x_0^-)\leq 0$, and we will also arrive at a contradiction by using the equations \eqref{contc} or \eqref{contq}.

The situation is slightly different if this strictly negative minimum is achieved at $Pc(x_0)$ at the points $x_0=\pm2$ or $x_0=0$. In $x_0=0$ we have that $c(0^-)=(v_3/v_2)c(0^+)$ and then the absolute minimum is achieved at $Pc(0^-)$, since $v_3/v_2\ge 1$, and necessarily $c_x(0^-)\le 0$ and the same contradiction is obtained by looking at \eqref{contc}. At the point $x_0=\pm2$ the situation is similar (remember that $c(2^-)= (v_1/v_4) c(-2^+)$  and $v_1/v_4\ge1$), and $c_x(2^-)\le 0$. Because of the same equation \eqref{contc}, we also arrive at a contradiction.

To extend the positivity of $(-A+sI)^{-1}$ to general $(\varphi_1,\varphi_2)^T\in X$, one can approximate in the $L^2$ norm $(\varphi_1,\varphi_2)^T$ by continuous functions $(\varphi_1^k,\varphi_2^k)^T$ and use that the $L^2$ convergence of
$(-A+s\Id)^{-1}(\varphi_1^k,\varphi_2^k)^T$ implies pointwise convergence almost everywhere in $x$.
\end{pf}
\begin{lemma}\label{lem:posop}
$e^{tA}$ is also a positive operator from $X$ to $X$ for $t\geq 0$.
\end{lemma}
\begin{pf}
We use the formula from K.J. Engel and R. Nagel (\cite{EngelNagel}, Corollary IV.2.5):
$$\exp(tA)\begin{pmatrix}c\\q\end{pmatrix}=\lim_{n\to\infty}\left(\Id-\dfrac{t}{n}A\right)^{-n}\begin{pmatrix}c\\q\end{pmatrix}=
\lim_{n\to\infty}\left(\dfrac{n}{t}\right)^{n}\left[\left(-A+\dfrac{n}{t}\Id\right)^{-1}\right]^n\begin{pmatrix}c\\q\end{pmatrix},
$$
for every $t>0$. It follows from Lemma \ref{pos} that if $(c,q)^T\in X_+$ then $\exp(tA)(c,q)^T\in X_+$.
\end{pf}
\begin{lemma}
For $v_1>v_2$ and $v_3>v_4$, the spectral radius $\r_0:=\r((-A+\Id)^{-1})$ is strictly positive, with a (non-optimal) lower bound
\begin{equation*}
\r_0 \geq \ \dfrac{Q^p}{(R+1)Q^p-RP v^{min}} >0
\end{equation*}
with
\begin{equation}\label{eq:Qp}
 Q^p > \  Q_0 :=\dfrac{1 }{RP}\left(v^{min} \,\dfrac{v^{max}-v^{min}}{2}+v^{max}\left(RP^2+1\right)  \right)  >0
\end{equation}
and $v^{max}:=\max\{v_1,v_2,v_3, v_4\}$ and $v^{min}:=\min\{v_1,v_2,v_3, v_4\}$.
\end{lemma}
\begin{pf}
We are going to use the following general property of positive linear bounded operators $T$ in ordered Banach spaces:  if for some real $\varepsilon>0$ the resolvent operator $(\varepsilon\Id-T)^{-1}$ either does not exist or is not positive, then the spectral radius $\r(T)$ is larger than or equal to $\varepsilon$.

This general property can be found as a particular case of Prop. 4.4.1 of P. Meyer-Nieberg \cite{Meyer} (p. 248) or of Cor. 1.3 of R. Nagel \cite{Nagel}, in both cases in the context of Banach Lattices. Let us recall the argument of the proof, just to show that it is independent of the Banach Lattice structure.

One argues by contradiction. It is a well-known fact that if $\delta>\r(T)$ then there exists a new equivalent norm $\|\cdot\|'$ of the Banach space such that $\r(T)<\|T\|'<\delta$ (see for example H.M. Rodrigues and J. Sol\`a-Morales \cite{Hildeb}). Therefore $(\delta\Id-T)^{-1}$ exists because the following power series in $x=T/\delta$
$$(\delta\Id-T)^{-1}=\dfrac{1/\delta}{\Id-T/\delta}=\sum_{j=0}^\infty \delta^{-j-1}T^j$$
is convergent in the operator norm $\|\cdot\|'$, since $\|T\|'/\delta <1$, and converges to a positive operator.

We proceed with this general property in mind, to be applied to $T=(-A+\Id)^{-1}$. Let us consider the particular functions $c^p(x)=v_4+\frac{v_3-v_2}{2}(x+2)$ for $-2\le x\le 0$, and $c^p(x)=v_2+\frac{v_1-v_2}{2}x$ for $0\le x\le 2$, and the function $q^p(x)\equiv Q^p>0$, a constant to be chosen below. The functions $c^p(x),q^p(x)$ are strictly positive and satisfy the continuity and jump conditions \eqref{bcn0}  and therefore $(c^p(x),q^p(x))^T\in X_+\cap D(A)$. We have that $v^{min}\le c^p(x)\le v^{max}$.


Because of Lemma \ref{pos} with $s=1$ we know that the operator $(-A+\Id)^{-1}$ is positive. On the other hand, we observe now that
\begin{equation}\label{eqn3x}
\left.\left((-A+\Id)\begin{pmatrix} c^p\\q^p\end{pmatrix}\right)\right|_{I_i}=
\left( \begin{array}{cc}
\underbrace{v_i c^p_x+(RP^2+1)c^p-RPQ^p}_{ < 0\text{ if } Q^p \text{ large enough}}\\
\underbrace{-q^p_x}_{=0}\underbrace{-RPc^p+(R+1)Q^p}_{ > 0 \text{ if }Q^p \text{ large enough}}
\end{array}
\right) =:
\left( \begin{array}{ll} \varphi^p_1{ <0}\\ \\ \varphi^p_2{ >0}\end{array}\right)
\end{equation}
So, we will look for a $Q^p$ large enough to fulfil \eqref{eqn3x} at all $x$ at some $I_i$. More concretely, we will choose an interval $I_i$ in which $v_i=v_{min}$. The monotonicity of the functions involved makes $Q^p$ large enough, which more precisely means that
$$Q^p > \, \max\left\{ \dfrac{RP}{R+1} v^{max}, \dfrac{1 }{RP}\left(v^{min} \,\dfrac{v^{max}-v^{min}}{2}+v^{max}\left(RP^2+1\right)  \right) \right\}$$
which, indeed, is
$$ Q^p > Q_0:=\dfrac{1 }{RP}\left(v^{min} \,\dfrac{v^{max}-v^{min}}{2}+v^{max}\left(RP^2+1\right)  \right)  >0.
 $$
Then
\begin{equation*}
-(-A+\Id)^{-1}\left( \begin{array}{ll} -\varphi^p_1 > 0\\-\varphi^p_2 < 0\end{array}\right)=\begin{pmatrix} c^p>0\\q^p>0\end{pmatrix}.
\end{equation*}
Therefore, for $\varepsilon>0$
\begin{equation}\label{eqn5x}
[\varepsilon\Id-(-A+\Id)^{-1}]\left( \begin{array}{ll} -\varphi^p_1\\-\varphi^p_2\end{array}\right)=
\begin{pmatrix} c^p\\q^p\end{pmatrix}+\varepsilon
\left( \begin{array}{cc}
-v_i c^p_x-(RP^2+1)c^p+RPQ^p\\
q_x^p+RPc^p-(R+1)Q^p
\end{array}
\right)= :
\left( \begin{array}{cc}
\psi_1^p\\
\psi_2^p
\end{array}
\right),
\end{equation}
and these two functions $\psi_1^p,\psi_2^p$ are also positive if $\varepsilon$ is small enough, depending on the choice of $Q^p$. How small $ \varepsilon$ has to be is determined only by the second component of \eqref{eqn5x}, since in the first component the factor that multiplies $\varepsilon$ is $-\varphi_1^p$, which is positive. Then, it is easily seen that the condition is
\begin{equation*}
\varepsilon<\varepsilon^p=\varepsilon^p(Q^p):= \dfrac{Q^p}{(R+1)Q^p-RP v^{min}},
\end{equation*}
(observe that the previous denominator is always strictly positive).


Therefore, we have a positive and compact operator $T=(-A+\Id)^{-1}$, such that its resolvent operator $R((-A+\Id)^{-1},\lambda)=\left(\lambda\Id-(-A+\Id)^{-1}\right)^{-1}$, if it exists, is not positive when
$\lambda=\varepsilon$ because it maps $(\psi_1>0,\psi_2>0)^T$ to $(-\varphi^p_1>0,-\varphi^p_2<0)^T$. And by the general considerations made at the beginning of this proof, necessarily
\begin{equation}\label{eq:eps0r0}
0<\varepsilon < \varepsilon^p\le \r_0={\rm r}((-A+\Id)^{-1}).
\end{equation}
\end{pf}
\begin{pf}({\bf of Theorem \ref{thfa}})
Because of the previous results in the above four lemmas in this section, that is, compactness, positivity, and strict positivity of the spectral radius, we can apply the Theorem of Krein-Rutman (see M.G. Krein and M.A. Rutman \cite{KR}, or Theorem 4.1.4 of \cite{Meyer}) to the operator $T=(-A+\Id)^{-1}$ in $X$ relative to the closed convex positive reproducing cone $X_+$. As a consequence, its spectral radius $\r_0$ is an eigenvalue of $T$ with a non-negative eigenfunction $(c_0,q_0)^T$, that is $(-A+\Id)^{-1}(c_0,q_0)^T=\r_0(c_0,q_0)^T$, $(-A+\Id)(c_0,q_0)^T=(1/\r_0)(c_0,q_0)^T$, and $A(c_0,q_0)^T=(1-(1/\r_0))(c_0,q_0)^T$. Observe that if $\mu$ is an eigenvalue of $T$, then $\lambda=1-1/\mu$ is an eigenvalue of $A$, and with the same eigenfunctions. Hence,
$$\Re (\lambda) = 1-  \dfrac{\Re(\mu)}{|\mu|^2} \leq 1-\dfrac{1}{\r_0}.$$
It must be fulfilled that the largest real eigenvalue is indeed $\lambda_0=1-\frac{1}{\r_0}$, because:  as $\r_0$ is an eigenvalue of $T$, this $\lambda_0$ is an eigenvalue of $A$; and is the largest one because  $\Re (\lambda)\leq \lambda_0$, which we know is strictly negative by Proposition \ref{prop:Reneg} (when inequalities \eqref{ineq} hold in the strict sense).

The relation between $\lambda_0$ and $Q^p$ comes from the inequality  $\varepsilon <\varepsilon^p\le\r_0$ (see \eqref{eq:eps0r0}). Consequently, $\lambda_0>1-\dfrac{1}{\varepsilon^p}=:- M^p$, where
\begin{equation}\label{eps00}
-M^p = -M^p(Q^p) := -R+\dfrac{RPv^{min}}{Q^p} <0
\end{equation}
which is strictly increasing in $Q^p$. Also, observe that
\begin{equation}\label{eq:Rlow}
-R \leq -M^p(Q^p)< \lambda_0
\end{equation}
which, indeed, is achieved in \eqref{eps00} if we take $Q^p=+\infty$.

If we take the best possible $Q^p$ in \eqref{eq:Qp}, that is, $Q^p\to Q_0 =\dfrac{1 }{RP}\left(v^{min} \,\dfrac{v^{max}-v^{min}}{2}+v^{max}\left(RP^2+1\right)  \right)$, we then have
$$-M^p(Q^p) \to -M_0 := -R+\dfrac{R^2P^2v^{min}}{v^{min}\frac{v^{max}-v^{min}}{2} +v^{max}(RP^2+1)} \leq \lambda_0.$$

Recall that there are two different statements under the name of Theorem of Krein-Rutman (even in the original paper \cite{KR}), one weaker than the other. We are using the weak one, where the operator $T$ is positive, but not strictly positive. So, now it still remains to show that $c_0,q_0$ are not only non-negative functions, but also strictly positive.

We recall that the eigenfunction equations are
\begin{equation}\label{eqn0f}
\left\{ \begin{array}{rll}
\lambda_0 c_0+v_i c_{0,x}&=\,RP(-Pc_0+q_0)\quad &\hbox{for}\: x\in I_i, \; i=1,2,3,4 \\
\lambda_0 q_0-q_{0,x}&=\,R(Pc_0-q_0)\quad &\hbox{for}\: x\in I_i, \; i=1,2,3,4,
\end{array}\right.
\end{equation}
with the boundary conditions \eqref{bcn0}.

We will argue by contradiction. Suppose that at some point $x_0\in I_i$ one has $c_0(x_0)=q_0(x_0)=0$. Then, by existence and uniqueness of ordinary differential equations, one would have $c_0(x)\equiv q_0(x)\equiv 0$ in the whole $I_i$, and by the boundary conditions, this would extend to the whole of $[-2,2]$. Contradiction.

Suppose now that $c_0(x_0)=0$ and $q_0(x_0)>0$. By the first equation in \eqref{eqn0f} we see that $c_{0,x}(x_0)>0$. This would imply $c_0(x)<0$ for $x<x_0$, and this would be a contradiction if $x_0$ is in the interior of $I_i$ or is its rightmost endpoint. If it is the leftmost endpoint of $I_i$, the boundary conditions would imply that $c_0(x)$  would also vanish at the rightmost endpoint of $I_{i-1}$. Contradiction again.

Suppose now that $c_0(x_0)>0$ and $q_0(x_0)=0$. The argument is then analogous, but after the second equation in \eqref{eqn0f} instead of the first. This finishes the proof of part {\it (i)} of Theorem \ref{thfa}.

To prove part {\it (ii)} of Theorem \ref{thfa} we observe that since $e^{tA}(c_0(x),q_0(x))^T= e^{t\lambda_0}(c_0(x),q_0(x))^T$, and $e^{tA}$ is a positive operator (see Lemma \ref{lem:posop}), the only thing we need is that $(|c(x,0)|,|q(x,0)|)^T\le M(c_0(x),q_0(x))^T$ for some $M>0$. And this is always possible since $c_0(x),q_0(x)$ are strictly positive functions.

Finally, to prove part {\it (iii)} of Theorem \ref{thfa}, observe that $\lambda=0$ is an eigenvalue in the limit case   $v_1=v_2=v_3=v_4$ (see Remark \ref{rem:limcase}). So, from Proposition \ref{prop:Reneg} we can say that $\lambda_0=0$ is the dominant eigenvalue. Also from Remark \ref{rem:limcase} below, we know that the corresponding eigenfunction is $(1,P)^T$ in this case, so it is strictly positive. So, the same argument as above can be used to prove \eqref{final} in this limit case.
\end{pf}
\section{Sensitivity Analysis: variation of $\lambda_0$ with respect to the parameters}\label{sens}
In this section, we will obtain the formula for the derivative of the dominant eigenvalue $\lambda_0$ or of any other eigenvalue $\lambda$ with respect to the six different parameters of the problem. To do so, we are going to use the Adjoint Method, which we explain in the present section. Therefore, we start computing the adjoint operator $A^*$ and its domain $D(A^*)$.

We consider the following operator as the adjoint one:
\begin{equation}\label{eqnadj}
\left.\left(A^*\begin{pmatrix} c^*\\ q^*\end{pmatrix}\right)\right|_{I_i}=A_i^*\left(\left.\begin{pmatrix} c^*\\ q^*\end{pmatrix}\right|_{I_i}\right)= \begin{pmatrix}  v_i\dfrac{d}{dx}-RP^2\, Id & RP\,Id\\ RP\,Id & -\dfrac{d}{dx}-R\,Id\end{pmatrix} \begin{pmatrix} c^*\\ q^*\end{pmatrix}
\end{equation}
for $I_i=[x_i,x_{i+1}],\ i = 1,2,3,4$, with domain
$$\mathcal{D}(A^*)= \left\{  (c^*,q^*)\in H^1(-2,2)\times H^1(-2,2)\, |\, (c^*,q^*) \textrm{ fulfil } \eqref{bcnadj} \right\}$$
with
\begin{equation}\label{bcnadj}
\left\{ \begin{array}{ll}
  &c^*(x_i^-)=c^*(x_i^+) \quad \hbox{for}\: x_i=0,\pm 2,\\
  &v_1c^*(-1^-)=v_2c^*(-1^+),\\
  &v_3c^*(1^-)= v_4c^*(1^+), \\
  &q^*(x_i^-)=q^*(x_i^+) \quad \hbox{for}\: x_i=-1,0,1,\pm 2 \hbox{. ($2^+$ means $-2^-$)}.
\end{array}\right.
\end{equation}
We can check that, indeed, $\langle A (c_1,q_1)^T, (c_2,q_2)^T\rangle_X = \langle (c_1,q_1)^T,A^*(c_2,q_2)^T\rangle_X$.
Remember that the eigenvalue problem for $A$ for a certain eigenvalue $\lambda$ is given in \eqref{eqn0}-\eqref{bcn0}.
And the eigenvalue problem for $A^*$ for the eigenvalue $\overline{\lambda}$ can be written in the same form as
\begin{equation}\label{two}
\left\{ \begin{array}{rll}
\overline{\lambda} c^*-v_i c^*_x&=RP(-Pc^*+q^*)&\quad \hbox{for}\: x\in I_i, \; i=1,\dots 4, \\
\overline{\lambda} q^*+q^*_x&=R(Pc^*-q^*)&\quad \hbox{for}\: x\in I_i, \; i=1,\dots 4,
\end{array}\right.
\end{equation}
with the boundary conditions \eqref{bcnadj}.  In Fig. \ref{first} (right), we can see a plot of the eigenfunction of $A^*$, with formula \eqref{eq:fupAadj}, corresponding to the dominant eigenvalue $\lambda_0$ in the case study of Section \ref{case}.

%
%
Our objective is to obtain an expression for the partial derivatives of an eigenvalue $\lambda$
with respect to the six parameters of the problem. To do so, we will use the Adjoint Method, following the ideas of Chpt. IV.7 of G. Ioos and D.D. Joseph \cite{Ioos}. We will apply this to the dominant eigenvalue $\lambda_0$, but the deduction is for a general eigenvalue $\lambda$.

We start deriving with respect to the velocities $v_k$. This is slightly more complicated than the other parameters, since the changes in the velocities make changes in the domain $D(A)$.
We fix $k\in\{1,2,3,4\}$, and we differentiate the eigenvalue equations \eqref{eqn0} and \eqref{bcn0} with respect to $v_k$. To simplify, let us denote by $'$ the derivatives with respect to $v_k$. We obtain the following first variation problem:
\begin{equation}\label{eqn0d}
\left\{ \begin{array}{ll}
&v_i c'_x=-\lambda c'-PR(Pc'-q')-\lambda'c-\delta_{i,k}\ c_x; \\
&q'_x=\lambda q'-R(Pc'-q')+\lambda'q,
\end{array}\right.
\end{equation}
for $x\in I_i$, $i=1,2,3,4$, with the boundary conditions
\begin{equation}\label{bcn0d}
\left\{
\begin{array}{ll}
&c'(x_i^+)=c'(x_i^-), \qquad\textrm{ for }x_i=\pm 1\\
&\delta_{1,j}\ c(-2^+)+v_1 c'(-2^+)=\delta_{4,j}\ c(2^-)+v_4c'(2^-)\\
&\delta_{3,j}\ c(0^+)+v_3 c'(0^+)=\delta_{2,j}\ c(0^-)+v_2 c'(0^-)\\
&q'(x_i^-)=q'(x_i^+), \qquad\textrm{ for } x_i=\pm 2, \pm 1, 0.
\end{array}
\right.
\end{equation}
Here $(c,q)$ is the solution of \eqref{eqn0}-\eqref{bcn0}, and $\delta_{i,k}$ is the Kronecker's delta. Observe that we now have a non-homogeneous problem for $c', q'$ and $\lambda'$, both in \eqref{eqn0d} and in \eqref{bcn0d}, once  $(c,q)^T$ is supposed to be given.

Now we multiply the two equations of \eqref{eqn0d}, by the functions $\overline{c^*}$ and $\overline{q^*}$ of \eqref{two}, respectively, subtract the results, and use \eqref{two}, obtaining
\begin{equation*}
v_i(c'\bar{c}^*)_x-(q'\bar{q}^*)_x=-\lambda'(c\bar{c}^*+q\bar{q}^*)-\delta_{i,k}c_x\bar{c}^*.
\end{equation*}
Now, integrating between $x_i$ and $x_{i+1}$, summing these integrals for $i=1,2,3,4$, and using the boundary conditions \eqref{bcn0}, \eqref{bcnadj} and \eqref{bcn0d}, one gets
$$
\sum_{i=1}^4 \int_{x_{i}}^{x_{i+1}}\left(\lambda'(c(s)\bar{c}^*(s)+q(s)\bar{q}^*(s))+\delta_{i,k}c_x(s)\bar{c}^*(s)\right)\, ds=
\left\{\begin{matrix}
-c(x_k^+)\bar{c}^*(x_k^+), \text{ if } k \text{ is odd}\\
c(x_{k+1}^-)\bar{c}^*(x_{k+1}^-), \text{ if } k \text{ is even }
\end{matrix}
\right\}.
$$
Isolating $\lambda'$, we obtain
\begin{equation}\label{lapri1}
\dfrac{\partial\lambda}{\partial v_k}=
\dfrac{-\delta_{1,k}c(-2^+)\bar{c}^*(-2^+)+\delta_{2,k}c(0^-)\bar{c}^*(0^-)-\delta_{3,k}c(0^+)\bar{c}^*(0^+)+\delta_{4,k}c(2^-)\bar{c}^*(2^-)- \int_{x_k}^{x_{k+1}}c_x\bar{c}^*}
{\int_{-2}^2(c\bar{c}^*+q\bar{q}^*)}.
\end{equation}

As an interpretation of \eqref{lapri1}, we observe that each interval $I_i$ in the liquid phase has an inlet boundary point ($x_i$ if $i$ is odd, $x_{i+1}$ if $i$ is even) and an exit one (the other one). The boundary term in the numerator of \eqref{lapri1} is evaluated in one of these inlet ports, with a sign depending on whether the port is at the left or right side of the interval.

The same ideas can be used to obtain similar formulas for the derivative of $\lambda$ with respect to $R$ and to $P$. If $'$ means now derivatives with respect to $R$, the first variation equation now reads

\begin{equation*}
\left\{ \begin{array}{ll}
&v_i c'_x=-\lambda c'+RP(-Pc'+q')+P(-Pc+q)-\lambda'c \quad \hbox{for}\: x\in I_i, \; i=1,\dots 4, \\
&q'_x=\lambda q'-R(Pc'-q')-(Pc-q)+\lambda'q\quad \hbox{for}\: x\in I_i, \; i=1,\dots 4,
\end{array}\right.
\end{equation*}
with the same boundary conditions \eqref{bcn0} for $c'$ and $q'$, since $R$ does not appear in the boundary conditions. Multiplying the first equation by $\bar{c}^*$ and the second by $\bar{q}^*$, subtracting the two equations, and integrating by parts this difference, by using the boundary conditions one gets
$$
0=\int_{-2}^2 (P(-Pc+q)\bar{c}^*+(Pc-q)\bar{q}^*)-\lambda'\int_{-2}^2 (c\bar{c}^*+q\bar{q}^*)
$$
and
\begin{equation}\label{lapriR}
\dfrac{\partial\lambda}{\partial R}=
\dfrac{\int_{-2}^2 (P(-Pc+q)\bar{c}^*+(Pc-q)\bar{q}^*)}{\int_{-2}^2 (c\bar{c}^*+q\bar{q}^*)}=
-\dfrac{\int_{-2}^2 (Pc-q)(P\bar{c}^*-\bar{q}^*)}{\int_{-2}^2 (c\bar{c}^*+q\bar{q}^*)}.
\end{equation}
Proceeding as in the last case, one also obtains the derivative of $\lambda$ with respect to $P$:
\begin{equation}\label{lapriP}
\dfrac{\partial\lambda}{\partial P}=
\dfrac{\int_{-2}^2 (R(-2Pc+q)\bar{c}^*+Rc\bar{q}^*)}{\int_{-2}^2 (c\bar{c}^*+q\bar{q}^*)}.
\end{equation}
All these derivatives will be computed for the particular case study of Section \ref{case}: see the values of the derivatives in \eqref{eq:dlamdvi} and \eqref{eq:dlamdRdP} (whose complete calculations are derived in Appendix \ref{sec:derivcasestudy}). As well as the needed dominant eigenfunctions, both of $A$ and $A^*$ (see Fig. \ref{first}), whose complete calculations can be found in Appendix \ref{sec:Cji}.
%
%
\section{A limit case: equal velocities}\label{limit}
As we have seen in Remark \ref{rem:dissipE}, one of the two phenomena responsible for the dissipation comes from the conditions  $v_1>v_2$ and $v_3>v_4$, meaning that the two extraction ports $x=-1$ and $x=1$ are active. The other is  $Pc(x)\not\equiv q(x)$, meaning that the two concentrations are not in equilibrium. It is interesting to study the limit case when these two inequalities of \eqref{ineq} become equalities, namely when $v_1=v_2$ and $v_3=v_4$. In this case, the only way of energy dissipation (see \eqref{dEdt}) is the tendency to equilibrium between the two phases. Notice that in this case the four velocities must be equal: $v_1=v_2=v_3=v_4=v$. This is due to the other two inequalities in \eqref{ineq} (non-strict version), namely $v_3\ge v_2$ and $v_1\ge v_4$, meaning that $x=0$ and $x=\pm 2$ are injection ports.

In this case, the total mass of the solute, namely $Q(t)=\int_{-2}^2(c(x,t)+Pq(x,t))\ dx$ remains constant along solutions, since $\frac{d}{dt}Q(t)=(v_2-v_1)c(-1,t)+(v_4-v_3)c(1,t)$ vanishes identically, but $\frac{d}{dt}E(t)=-\int_{-2}^2 |Pc-q|^2\ dx$ (see \eqref{dEdt}), meaning that the only source of energy dissipation is the tendency to the equilibrium of phases, namely $Pc\equiv q$.

In this case, also, some of the previous calculations with the characteristic equation \eqref{eq:f} can be made in an easier way.
In fact, when $v_i=v$ for $i=1,\ldots,4$, the expression for the matrix $C(\lambda)$ in \eqref{eq:C} becomes simpler: $C(\lambda)=\exp(4F(\lambda))$, where $F(\lambda)$ is given in \eqref{eqn1v}. So, we can obtain more explicit formulas for the trace and determinant of $C(\lambda)$:
\begin{align*}
&{\rm trace}(e^{4F(\lambda)})  = e^{4\nu_1(\lambda)}+e^{4\nu_2(\lambda)}
                       = 2\exp\left( \dfrac{4\alpha(\lambda)}{v}\right)\cosh\left(4\sqrt{\left(\dfrac{\alpha(\lambda)}{v}\right)^2+\dfrac{\beta(\lambda)}{v}}\right)\\
&{\rm det}(e^{4F(\lambda)}) = e^{4\nu_1(\lambda)}e^{4\nu_2(\lambda)}
                     = \exp\left(\frac{8\alpha(\lambda)}{v}\right)
\end{align*}
where $\nu_1(\lambda),\nu_2(\lambda)$ are given in \eqref{eq:nu}, and $\alpha(\lambda),\beta(\lambda)$ are given in \eqref{alpha} and \eqref{eq:beta}, respectively.
Hence, we have
\begin{equation}\label{flim}
\Delta(\lambda)=\exp\left( \dfrac{4\alpha(\lambda)}{v}\right)\left( 2\cosh\left(4\sqrt{\left(\dfrac{\alpha(\lambda)}{v}\right)^2+\dfrac{\beta(\lambda)}{v}}\right)-\exp\left( \dfrac{4\alpha(\lambda)}{v}\right)\right)-1.
\end{equation}
Taking $\lambda=0$, and since $\beta(0)=0$, it is easy to see that
$$\Delta(0)=\exp\left( \dfrac{4\alpha(0)}{v}\right)\left(\exp\left( \dfrac{4\alpha(0)}{v}\right)+\exp\left( -\dfrac{4\alpha(0)}{v}\right)-\exp\left( \dfrac{4\alpha(0)}{v}\right)\right)-1$$
and, therefore, $\Delta(0)=0$.

The expression \eqref{flim} of $\Delta(\lambda)$ in the present case is interesting because it exhibits its structure in a much clearer way than the expression \eqref{eq:f} for the general case. We have to remember that the roots of the equation $\Delta(\lambda)=0$ are the eigenvalues of the operator $A$. This is also true for the non-real ones, which are also of interest.

In the present case, we can have an even more explicit expression for the eigenvalues. Remember that the roots of $\Delta(\lambda)=0$ are the values of $\lambda\in\C$
for which the matrix $C(\lambda)=\exp(4F(\lambda))$ has the eigenvalue $1$. This is equivalent to saying that $F(\lambda)$ admits an eigenvalue of the form  $i\pi k/{2}$, $k\in\Z$, or, again equivalently
$$
\det
\begin{pmatrix}
-\dfrac{\lambda+P^2R}{v}-\dfrac{i\pi}{2}k&\dfrac{PR}{v}\\-PR&\lambda+R-\dfrac{i\pi}{2}k
\end{pmatrix}
=0
$$
or
$$\lambda^2+\left(R(1+P^2)+\dfrac{i\pi k}{2}(v-1)\right) \lambda+
\dfrac{i\pi k}{2}R\left(v-P^2\right)+v\dfrac{\pi^2 k^2}{4}=0.
$$
That means
\begin{equation}\label{eq:lambdapm}
\lambda_k^{\pm}=\dfrac{-1}{2} \left(R(1+P^2)+\dfrac{i\pi k}{2}(v-1)\right) \pm
\dfrac{1}{2}\sqrt{\left(R(1+P^2)+\dfrac{i\pi k}{2}({v-1})\right)^2-2i\pi k R\left(v-{P^2}\right)-v\pi^2 k^2}
\end{equation}
for $k\in\mathbb{Z}$.
Separating real and imaginary parts, this can also be written as
\begin{equation}\label{eq:lamReIm}
\lambda_k^{\pm}= \left( -\frac{R(1+P^2)}{2} \pm \frac{U(k)}{2}\right) \pm \left( -\frac{\pi k}{4}(v-1) + \frac{V(k)}{2}\right) i
\end{equation}
where
\begin{equation}\label{eq:UV}
U(k) = \sqrt{\frac{\sqrt{X(k)^2+Y(k)^2}+X(k)}{2}},
\qquad
V(k) = \operatorname{sign}(Y(k))\sqrt{\frac{\sqrt{X(k)^2+Y(k)^2}-X(k)}{2}}
\end{equation}
with
\begin{equation}\label{eq:XY}
X(k) =  - \frac{\pi^2 k^2}{4}(v+1)^2 +R^2(1+P^2)^2, \quad
Y(k) = \pi  R (P^2 - 1)(v+1)k.
\end{equation}
\begin{rem}
Observe that $\lambda_k^{\pm}$ and $\lambda_{-k}^{\pm}$ are complex conjugates of each other, for all $k\in\mathbb{Z}\setminus\{0\}$.
\end{rem}
\begin{rem}\label{rem:limcase} The expression \eqref{eq:lambdapm} for $k=0$ gives $\lambda_0^+=0$ and $\lambda_0^-=-R(1+P^2)$. Looking at the problem \eqref{eqn0}-\eqref{bcn0} we see that the eigenfunctions corresponding to the eigenvalues $\lambda=0$ and $\lambda=-R(1+P^2)$ are, respectively $(c(x),q(x))^T\equiv (1,P)^T$ (equilibrium of phases), and $(c(x),q(x))^T\equiv (-P,1)^T$.
\end{rem}
\begin{lemma}\label{lem:lambdaasymp}
For $|k|$ large enough, we have
\[
\lambda_k^+
\;=\;
\begin{cases}
\displaystyle
-R
\;+\;
\frac{4R^3(P^2-1)P^2}{\pi^2 (v+1)^2}\,\frac{1}{k^2}
\;+\;
i\left(
\frac{\pi}{2}\,k
\;-\;
\frac{2R^2P^2}{\pi (v+1)}\,\frac{1}{k}
\right)
\;+\;
O(k^{-3}),
& \text{if } P \geq 1, \\[1.5em]
\displaystyle
-RP^2
\;-\;
\frac{4R^3(P^2-1)P^2}{\pi^2 (v+1)^2}\,\frac{1}{k^2}
\;+\;
i\left(
-\frac{\pi v}{2}\,k
\;+\;
\frac{2R^2P^2}{\pi (v+1)}\,\frac{1}{k}
\right)
\;+\;
O(k^{-3}),
& \text{if } P \leq 1.
\end{cases}
\]
\vspace{0.5em}
\[
\lambda_k^-
\;=\;
\begin{cases}
\displaystyle
- R P^2
\;-\;
\frac{4 R^3 (P^2 - 1) P^2}{\pi^2 (v+1)^2}\,\frac{1}{k^2}
\;+\;
i\left(
-\frac{\pi}{2}\,k
\;+\;
\frac{2 R^2 P^2}{\pi (v+1)}\,\frac{1}{k}
\right)
\;+\;
O(k^{-3}),
& \text{if } P \geq 1,
\\[1.5em]
\displaystyle
- R
\;+\;
\frac{4 R^3 (P^2-1) P^2}{\pi^2 (v+1)^2}\,\frac{1}{k^2}
\;+\;
i\left(
\frac{\pi v}{2}\,k
\;-\;
\frac{2 R^2 P^2}{\pi (v+1)}\,\frac{1}{k}
\right)
\;+\;
O(k^{-3}),
& \text{if } P \leq 1.
\end{cases}
\]
\end{lemma}
\begin{pf}
We can compute the asymptotic expansion of $X(k),Y(k), U(k), V(k)$ and $\lambda_k^{\pm}$ in \eqref{eq:XY}, \eqref{eq:UV} and\eqref{eq:lamReIm}, when $|k|\to\infty$. The main ingredient is to use the asymptotic expansion of $\sqrt{1+x}$ for $x=|k|^{-1}$ small, and to take the dominant terms. Then, the above formulas follow.
\end{pf}
We can now prove the following properties for the eigenvalues. They can be seen in Fig. \ref{limitspect0}, for instance.
\begin{prop}\label{prop:properties}
Consider $\lambda_k^{\pm}$, $k\in\mathbb{Z}$, the eigenvalues given in \eqref{eq:lamReIm}. If $v_i=v$ for $i=1,\ldots 4$, we have the following properties:
\begin{enumerate}
\item $\lambda_0^-=-R(1+P^2)\leq \Re(\lambda_k^-)\leq\Re(\lambda_k^+)\leq \lambda_0^+=0$.
\item The sequence $\Re(\lambda_k^+)$ is decreasing in $|k|$ and tends to $\max\{-R,-P^2R\}$ when $|k|\to\infty$, and the sequence $Re(\lambda_k^-)$ is increasing in $|k|$ and tends to $\min\{-R,-P^2R\}$ when $|k|\to\infty$.
\item For $k$ large enough, the sequences $\Im(\lambda_k^{\pm})$ tend monotonically to $\pm\infty$ (if $P\geq 1$) or $\mp\infty$ (if $P<1$) as $k\to +\infty$, and the other way round when $k\to -\infty$.
\item For most values of $R,P,v$, we have $\Im(\lambda_k^\pm)\ne 0$, except for $k=0$. For exceptional values of these parameters, it can happen that $\Im(\lambda_k^{\pm})=0$, for $k=0$ and also for $k=\pm k^*$, for a certain $k^*\in\mathbb{Z}$.
\end{enumerate}
\end{prop}
\begin{rem}
$\Im(\lambda_k^{\pm})$ may not be monotone when $|k|$ is small (see, for instance, Fig. \ref{chev}).
\end{rem}
\begin{pf}
We start with the proof of the first property. From Proposition \ref{prop:Reneg} we know that $\Re(\lambda_k^{+})\leq 0$, where, from \eqref{eq:lamReIm}, $\Re(\lambda_k^{+}) = -\frac{R(1+P^2)}{2}+\frac{U(k)}{2}$. Hence, $R(1+P^2)+U(k)\leq 2R(1+P^2)$ and, therefore, $\Re(\lambda_k^{-})\geq -R(1+P^2)$. As $U(k)\geq 0$ (see formula \eqref{eq:UV}), we conclude point 1 of the present proposition.

To prove the monotonicity of $\Re(\lambda_k^{\pm})$ we first consider $P\neq 1$ and start by deriving $U(k)\geq 0$ with respect to $k$:
$$\dfrac{dU}{dk} = \dfrac{\pi (v+1) \left( -(X(k)+\sqrt{X(k)^2+Y(k)^2})\frac{\pi}{2}  (v+1) \, k+Y(k) R (P^2-1))\right)}{2\sqrt{2} \sqrt{X(k)^2+Y(k)^2} \sqrt{\sqrt{X(k)^2+Y(k)^2}+X(k)}} $$
Observe that, if $P\neq 1$, it is impossible that $X(k)=Y(k)=0$, as $Y(k)=0$ happens only when $k=0$, which would imply $X(k)>0$. Hence, $dU/dk$ is well defined if $P\neq 1$. Also observe that $dU/dk = 0$ only if $k=0$, which means that $U(k)$ is monotone for $k\in (-\infty,0)$ and for $k\in (0,\infty)$. Therefore, $\Re(\lambda_k^{\pm})$ is monotone for $k\in (-\infty,0)$ and for $k\in (0,\infty)$, and the local maximum or minimum of $\Re(\lambda_k^{\pm})$ are attained at $k=0$: $\Re(\lambda_0^{\pm})=\{ \lambda_0^-,\,\lambda_0^+\} =\{ -R(1+P^2),\, 0 \}$.

To see which is the type of monotonicity, we can use Lemma \ref{lem:lambdaasymp} to asymptotically differentiate $\Re(\lambda_k^{\pm})$ with respect to $k$:
\[
\left(\Re(\lambda_k^{\pm})\right)' = \mp \dfrac{8R^3|P^2-1|P^2}{\pi^2(v+1)^2}\dfrac{1}{k^3}+O\left(\dfrac{1}{k^4}\right).
\]
This means that $\Re(\lambda_k^+)$ is decreasing in $|k|$, and the sequence $Re(\lambda_k^-)$ is increasing in $|k|$, for $|k|$ sufficiently large and, hence, for all $k\in\mathbb{Z}$ (if $P\neq 1$). When $P=1$, we have $Y(k)=0$ and $X(k)= -{\pi^2 k^2}(v+1)^2/4+4R^2$, and, hence,
\[
U(k) =\begin{cases}
 \sqrt{X(k)} & \text{if } |k| <\dfrac{4R}{\pi (v+1)}, \\[7pt]
 0 & \text{if } |k| \geq\dfrac{4R}{\pi (v+1)}.
\end{cases}
\]
So,
\[
\Re(\lambda_k^{\pm}) =\begin{cases}
 -R\left( 1\mp\sqrt{1-\dfrac{\pi^2 k^2}{16 R^2}(v+1)^2}\right) & \text{if } |k| <\dfrac{4R}{\pi (v+1)}, \\[6pt]
 -R & \text{if } |k| \geq \dfrac{4R}{\pi (v+1)},
\end{cases}
\]
concluding the monotonicity of $\Re(\lambda_k^{\pm})$ also when $P=1$. From Lemma \ref{lem:lambdaasymp}, it is direct to see that
\[
\lim_{|k|\to\infty} \Re(\lambda_k^+)
=
\begin{cases}
 -R, & \text{if } P \geq 1, \\[6pt]
 -R P^2, & \text{if } P \leq 1 ,
\end{cases}
,\qquad
\lim_{|k|\to\infty} \Re(\lambda_k^-)
=
\begin{cases}
 -R P^2, & \text{if } P\geq 1, \\[7pt]
 -R, & \text{if } P\leq 1 ,
\end{cases}
\]
which finishes the proof of part 2 of this proposition.

To prove part 3, we now use Lemma \ref{lem:lambdaasymp} to asymptotically differentiate $\Im(\lambda_k^{\pm})$ with respect to $k$:
\[
\left(\Im(\lambda_k^{\pm})\right)' = \begin{cases}
 \pm \dfrac{\pi}{2} + O\left(\dfrac{1}{k^2}\right), & \text{if } P > 1, \\[6pt]
 \mp \dfrac{\pi}{2}v+ O\left(\dfrac{1}{k^2}\right), & \text{if } P < 1 ,
 \end{cases}
\]
And also from Lemma \ref{lem:lambdaasymp}, it is direct to see that.
\[
\lim_{k\to\pm\infty} \Im(\lambda_k^+)
=\begin{cases}
 \pm \infty, & \text{if } P \geq 1, \\[6pt]
 \mp \infty, & \text{if } P \leq 1 ,
\end{cases},\qquad \
\lim_{k\to\pm\infty} \Im(\lambda_k^-)
=\begin{cases}
 \mp \infty, & \text{if } P \geq 1, \\[6pt]
 \pm \infty, & \text{if } P \leq 1 ,
\end{cases},
\]
Finally, $\Im(\lambda_k^{\pm})=0$ is equivalent to proving that $\pi k (v-1)=2\,V(k)$. Squaring the previous equality, and using the definition of $V(k),X(k),Y(k)$ given in \eqref{eq:UV} and \eqref{eq:XY}, we can see that the following equality has to be fulfilled:
$$\frac{\pi^4 k^4(v-1)^4}{4}+ \pi^2 k^2 (v-1)^2 \left( -\frac{\pi^2 k^2}{4}(v+1)^2 + R^2 (P^2+1)\right)= \pi^2 R^2 (P^2-1)^2 (v+1)^2 k^2.
$$
For $k\neq 0$, this is equivalent to
\begin{equation}\label{eq:k}
k = \pm\dfrac{2R}{\pi (v-1)}\sqrt{\dfrac{2  \left( P^2v^2-(P^4+1)v +P^2 \right)}{ v }}.
\end{equation}
So, in order to have $\Im(\lambda_k^{\pm})=0$, the right-hand side of the previous equality has to be positive and an integer. This can happen in the following two scenarios: 1) when $P>1$, if $v<1/P^2$ or $v>P^2$, and a combination of parameters $R,P,v$ such that the right hand side of \eqref{eq:k} is an integer; 2) when $P<1$, if $v<P^2$ or $v>1/P^2$, and a combination of parameters $R,P,v$ such that the right hand side of \eqref{eq:k} is an integer. Otherwise, $\Im(\lambda_k^{\pm})\neq 0$ if $k\neq 0$. This proves the last point of the present proposition.
\end{pf}
\begin{rem}\label{rem:compact}
The second and third points of Proposition~\ref{prop:properties} show that $e^{At}$ cannot be compact (nor eventually compact) when all velocities are equal. Roughly speaking, since $\Re(\sigma(A)) \subset [-R(1+P^2),0]$, the spectrum of $e^{At}$ lies in an annulus with strictly positive inner radius. Hence, the spectrum of $e^{At}$ cannot approach $0$. As $0$ is the only possible accumulation point of the spectrum of a compact operator with infinitely many eigenvalues, this confirms that $e^{At}$ cannot be compact when all velocities coincide. Presumably, when the velocities are \textsl{close to equal}, a similar phenomenon may occur.
\end{rem}

In the following figures, we can see examples of the properties of Proposition \ref{prop:properties} being fulfilled.  The next Fig \ref{limitspect0} (left) is a plot in the complex plane of the first eigenvalues $\lambda_k^\pm$ for the parameters  $R = 18$, $P = 1.03\, (>1)$ as in the case study (see Section \ref{case}) and the velocity taken as $v=1.275$, which is the average of the four velocities
$1.53, 1.12, 1.43$ and $1.02$, which are the four velocities in the aforementioned case study. It is interesting and suggesting compare it with Fig. \ref{chev}, which is a numerical simulation of part of the spectrum of the operator $A$ for the values of the parameters in the case study. In Fig. \ref{limitspect0} (right), we can see the same picture but now for $P=0.5\, (<1)$, to see the other possible scenario in Lemma \ref{lem:lambdaasymp} and Proposition \ref{prop:properties}.

\begin{figure}[htpb]
  \centering
  \includegraphics[scale=0.4]{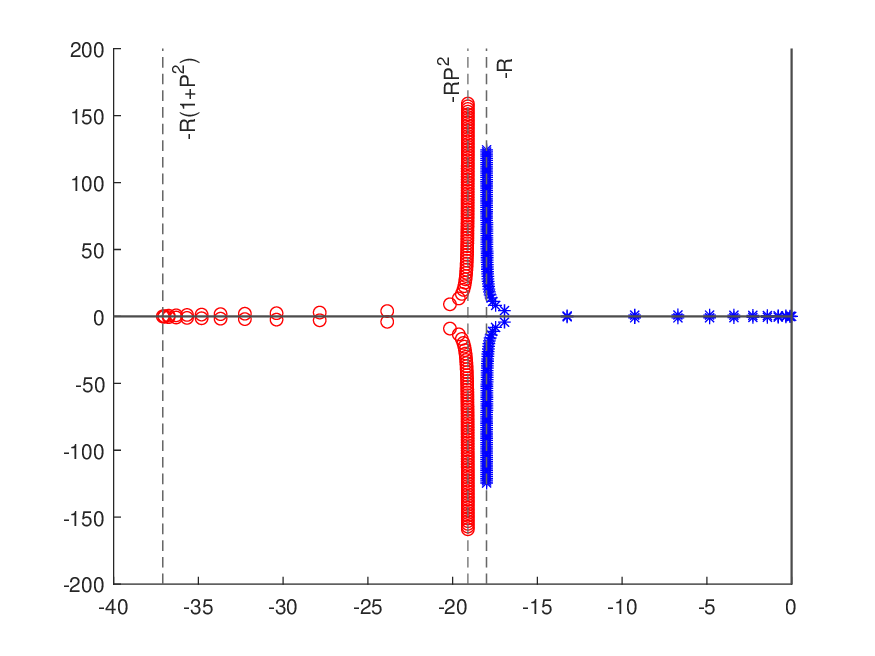}   \includegraphics[scale=0.4]{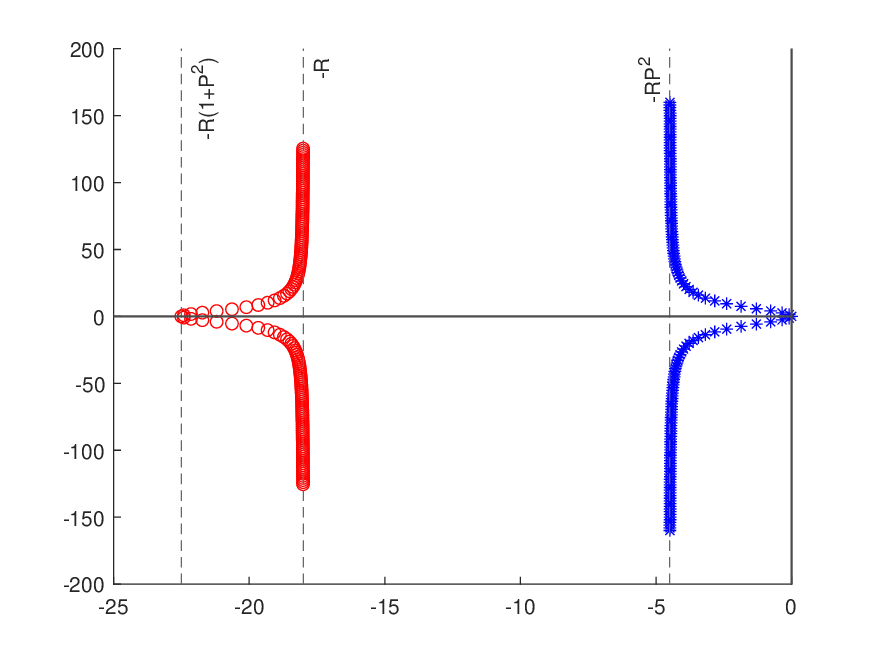}
\caption{Plot of first eigenvalues $\lambda^+$ (in blue-star) and $\lambda^-$ (in red-circle), $k=-80,\ldots 80$, for equal velocities $v=1.275$, and $R = 18$. Left: for $P=1.03\, (>1)$; right: $P=0.5\, (<1)$.} \label{limitspect0}
\end{figure}

%


\section{A case study: steady state, $\lambda_0$, dominant eigenfunction and spectrum, and sensitivity analysis}\label{case}
This section presents a case study illustrating the applicability of the results developed in this work. The practical example corresponds to the separation of the enantiomers of omeprazole. We use the physical data reported in \cite{Wei}, which describes this separation using an 8-column SMB.

The void fraction of the column is $\epsilon = 0.67$. The equilibrium constant of the enantiomers is $H_A = 2.14$ and $H_B = 2.96$ respectively, and the kinetic constant $k_A = 6.0 min^{-1}$ and $k_B = 5.0 min^{-1}$. The letters ``A'' and ``B'' designate the two enantiomers, A being the less retained by the solid phase and B the more retained one. In this brief study, we work only with the enantiomer A.

According to the so called ``triangle theory'' (see, for instance, G. Storti et al. \cite{Storti1993}, and A. Rajendran et al. \cite{Mazzotti}), to achieve the separation of the two species, the velocities $u_i$ in the different zones have to be selected such that $m_1>H_B>m_3>m_{2}>H_A>m_{4}$, where $m_i=\epsilon/(1-\epsilon) \times u_i/u_s, i=1,2,3,4$. Here we take the values
$$m_1=3.06,\quad m_{2}=2.24,\quad m_{3}=2.86 \quad\text{ and }\quad m_{4}=2.04.$$
Taking the length of each column as $L_{column} = 30$ cm and a switching time of $t_s=1.5$ min, the equivalent velocity of the solid phase is $u_s = 20$ cm/min. Then, the parameters of the dimensionless problem are (note that $L_{zone}=2L_{column}$)
$$R=k_AL_{zone}/u_s=18,\quad F=(1-\epsilon)/\epsilon=0.5,\quad P=\sqrt{FH_A}=1.03$$
$$v_1=1.53,\  v_{2}=1.12,\  v_{3}=1.43,\  \text{ and } v_{4}=1.02.$$

First, starting from these data that correspond to the species $A$, we can calculate the steady state of \eqref{eqn0t} when taking the feed rate $f_0=1$ in \eqref{bcn0t}. The result is depicted in Fig. \ref{steady}, which can be obtained in a similar procedure than the one described in Appendix \ref{sec:Cji}, but solving the corresponding 8--by--8 system.

\begin{figure}[htpb]
  \centering
  {\includegraphics[scale=0.4]{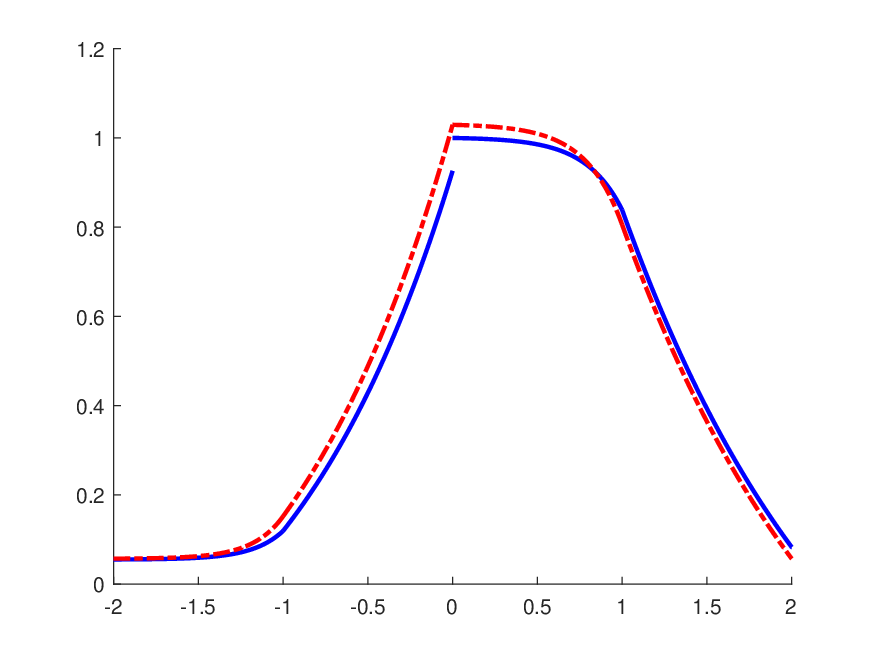}}
    \caption{Steady state for $(v_1,v_2,v_3,v_4,R,P)=(1.53, 1.12, 1.43, 1.02, 18,1.03)$ and $f_0=1$. In blue (solid line), $c$; in red (dash–dot line), $q$.}\label{steady}
\end{figure}

Now we use an initial value solver adapted to this type of problem (see J. Menacho et al. \cite{Menacho2011,Menacho2013,Menacho1}) to see the convergence to the steady state. Because of the linearity of the equations, we solve \eqref{eqn0t}-\eqref{bcn0t} with $f_0=0$, and then we see the convergence to zero. The numerical method is explained in more detail in Appendix \ref{sec:appendix3}, as well as in the above references. The resulting solution can be seen in Figure \ref{ivp}, with two different views of the time-dependent solution. Note the asymptotic profile of the solution for large values of $t$, and notice its coincidence with the profile given by the eigenfunction corresponding to the dominant eigenvalue (see Fig. \ref{first} (left), and Fig. \ref{fig:conv}).

\begin{figure}[htpb]
  {\includegraphics[scale=0.2]{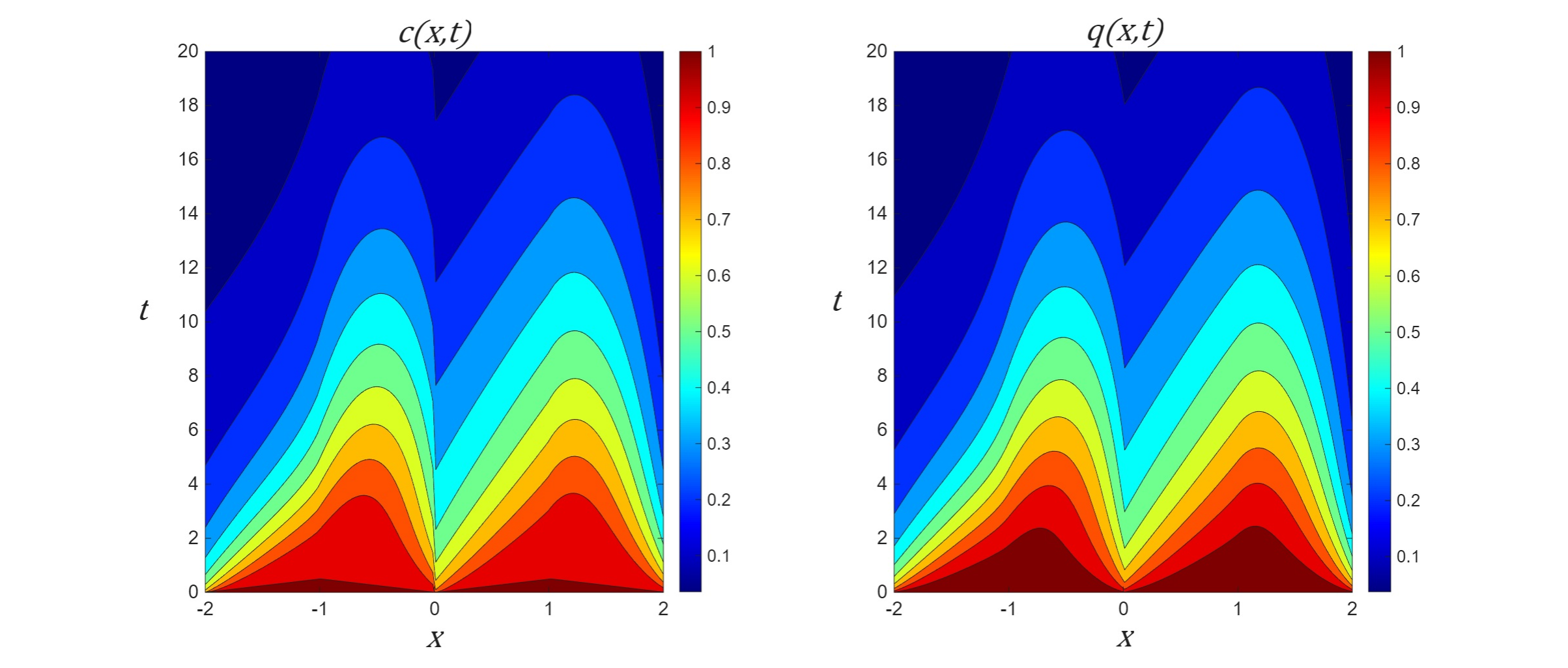}}
  \hspace{-1cm}
  {\includegraphics[scale=0.2]{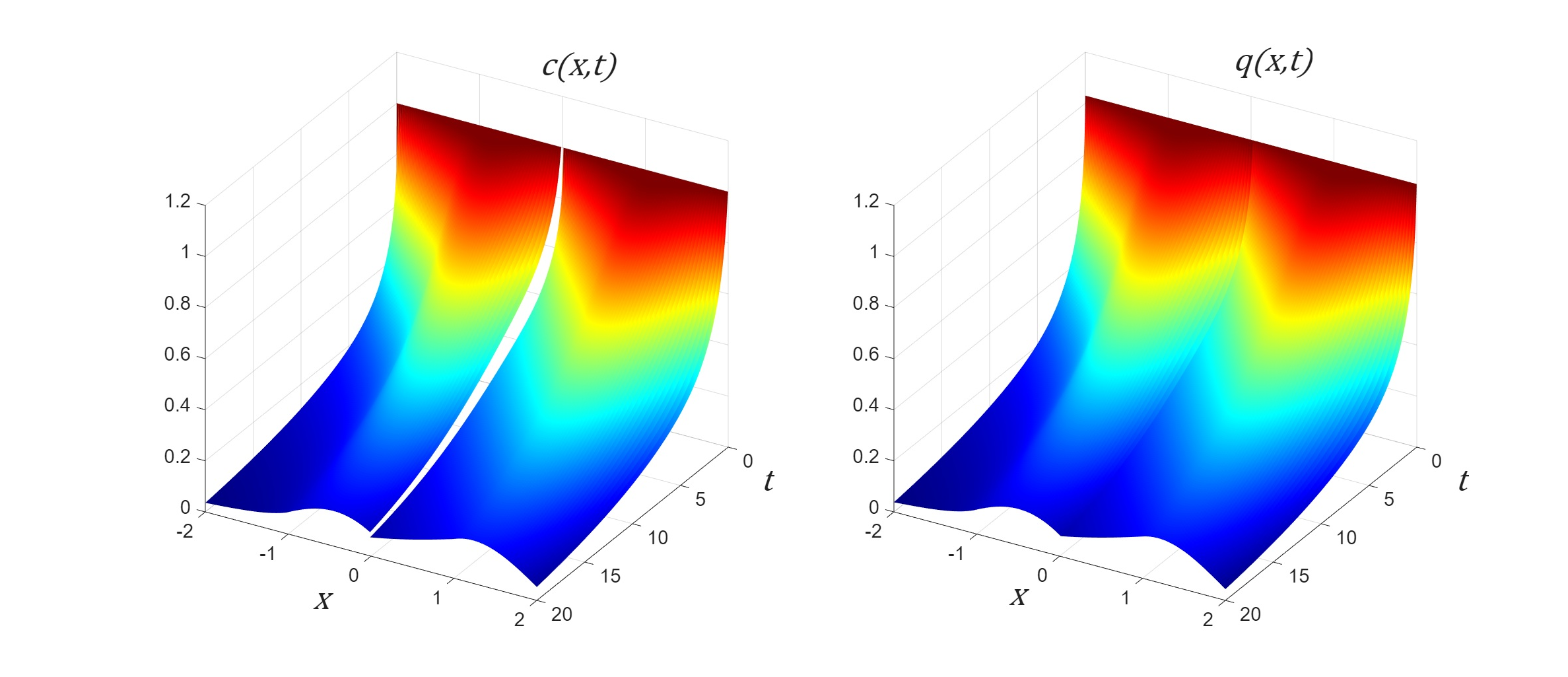}}
 \caption{Two views of the solution of the initial value problem with $f_0=0$, and initial conditions $(1,P)$, in the case study. First, contour plots of $c(x,t),q(x,t)$, respectively, at different times. Second, a view of the evolution of the solution, in terms of $t$ (and $x$). }\label{ivp}
\end{figure}

\begin{figure}[htpb]
  \centering
 \includegraphics[scale=0.5]{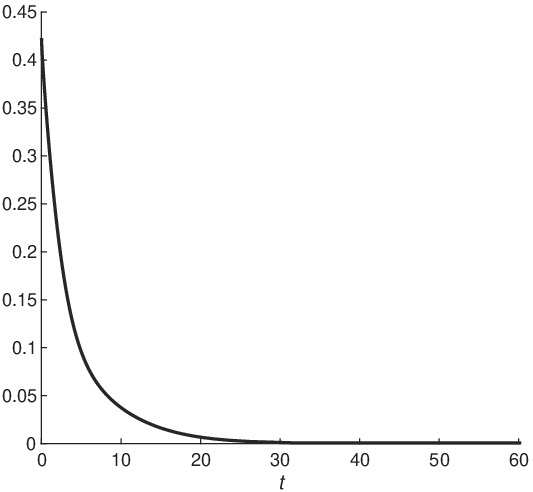}\hspace{2cm} \includegraphics[scale=0.5]{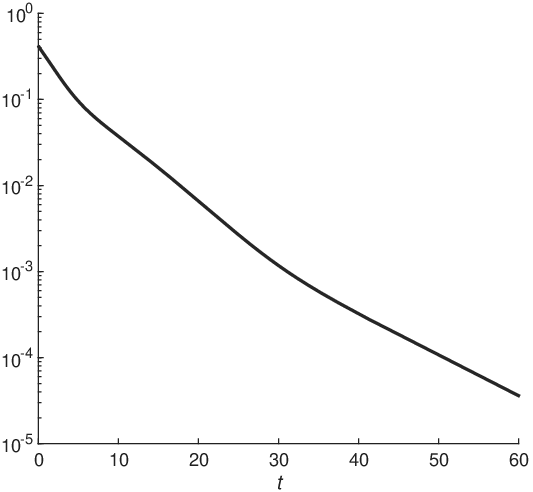}
  \caption{Left: time evolution of the root-mean-square difference between the solution and the mean-scaled eigenfunction, showing the convergence of the solution toward the eigenfunction profile. Right: same graphic, at logarithmic scale.}
  \label{fig:conv}
\end{figure}

Second, we want to compute $\lambda_0$, the dominant eigenvalue, in this case study. To numerically implement it, we proceed according to the following scheme:
\begin{enumerate}
\item We construct the matrix $C(\lambda)$ as the product defined in \eqref{eq:C} for the values of the parameters $v_i,R,P$ given above. To do this, we use the explicit formula for the matrices $M_i(\lambda)$ given in Section \ref{sec:Mi}.
\item To derive the characteristic equation $\Delta(\lambda)=0$ in \eqref{eq:f}, we recall that an explicit expression for $\det(C(\lambda))$ is provided in \eqref{detC}. However, to complete the construction of $\Delta(\lambda)$, we still need to compute the trace of the matrix $C(\lambda)$ obtained in the previous step.
\item Finally, we can find real roots of $\Delta(\lambda)=0$, being $\lambda_0$ the largest one. We can use, for instance, a bisection method, for which it is useful to know that $\lambda_0\in [-R,0]$ (see Theorem \ref{thfa} and \eqref{eq:Rlow}) or, even more precisely, that $\lambda_0\in [-M_0,0]$, where $M_0$ is given in \eqref{eps0} (see also Proof of Theorem \ref{thfa} in Section \ref{posit}).
\end{enumerate}
Note that $C(\lambda)$, as well as its determinant and trace, can be computed explicitly. But, despite dealing with 2-by-2 matrices, the computations are lengthy and intricate. That is why it is very convenient to use a symbolic or algebraic computation tool to carry out these calculations.

In the present case study, and with the help of \verb"Matlab" to implement the previous procedure, we have plotted the graph of the $\Delta(\lambda)$ in Fig. \ref{mult}, and we have obtained $\lambda_0 = -0.110377$ as the dominant eigenvalue. As $\lambda_0$ is the exponential coefficient for the dimensionless time \eqref{dimless}, this means that any transient of the chromatographic problem (the system \eqref{sysbasic} with $u=u_i$ for every zone, with boundary conditions \eqref{bcbasic}) will evolve as $e^{-t / \tau}$, where the time constant $\tau$, in this case, is $\tau=-L/(u_s \lambda_0)=13.5898$ minutes. That is, for $t=5\tau\approx68$ minutes, the magnitude of the concentration has been reduced by more than $99 \%$.

We can also plot the first eigenfunction of $A$ and the first eigenfunction of $A^*$, both corresponding to $\lambda_0 = -0.110377$. The first one is computed using formula \eqref{eq:fupA}, and the second one with formula \eqref{eq:fupAadj}. The computation of the corresponding coefficients $\X{C}{j}{i}$, $\Xe{C}{j}{i}$ is explained in detail in Appendix \ref{sec:Cji}, in general and also with the particular values of $v_i,R,P$ given above. The resulting eigenfunctions are given in Fig. \ref{first} below. Observe that when comparing Fig. \ref{first} (left) (eigenfunction of $A$) with Fig.  \ref{ivp}, we see the coincidence of the asymptotic profile with the profile of the eigenfunction. As we said in conjecture \eqref{conj}, this is something we believe is true, but that we have not been able to prove.

\begin{figure}[htpb]
  \centering
   {\includegraphics[scale=0.3]{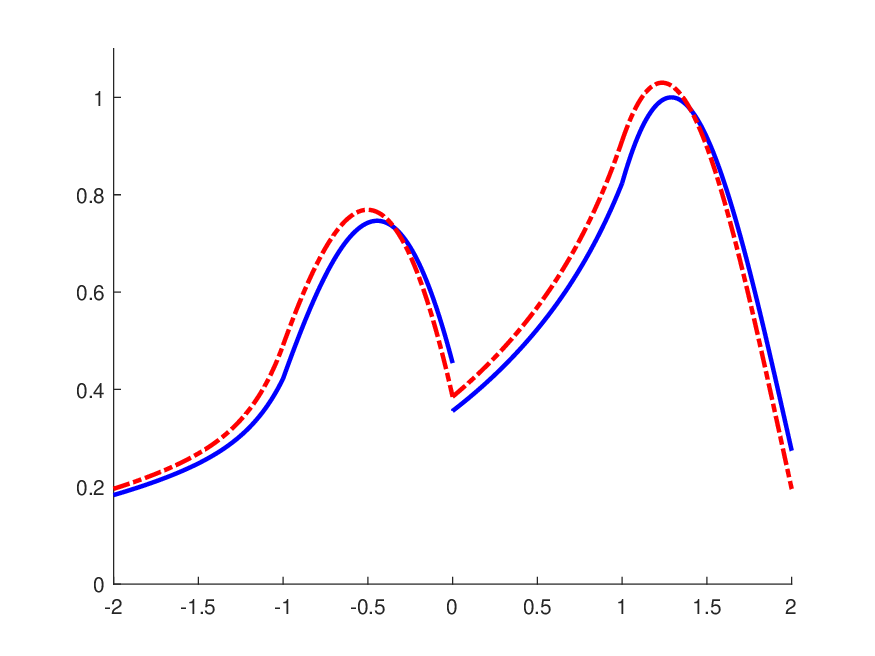}}  {\includegraphics[scale=0.3]{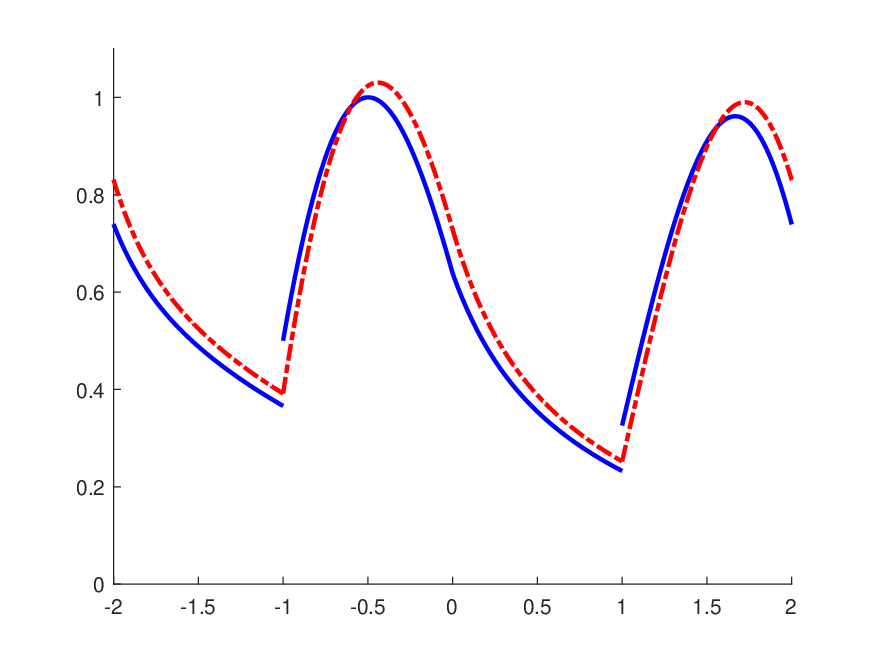}}
   \caption{Eigenfunction of $A$ (left) and of $A^*$ (right) for the dominant eigenvalue $\lambda_0$. Blue-solid line: $c(x),c^*(x)$; red dash-dot line: $q(x),q^*(x)$. The complete calculations can be found in Appendix \ref{sec:Cji}.}\label{first}
\end{figure}


Third, we perform the sensitivity analysis described in Section \ref{sens} for this case study. That is, we compute the derivatives of the dominant eigenvalue $\lambda_0$ with respect to the six parameters of the problem: $\frac{\partial\lambda_0}{\partial v_j}, \frac{\partial\lambda_0}{\partial R}, \frac{\partial\lambda_0}{\partial P}$ given in \eqref{lapri1}, \eqref{lapriR} and \eqref{lapriP}, respectively. The detailed procedure to implement the computation of these formulas is explained in Section \ref{sec:derivcasestudy} below. We have performed these computations in the case study using \verb"Matlab", obtaining the following values for each derivative:
\begin{equation}\label{eq:dlamdvi}
\dfrac{\partial\lambda_0}{\partial v_j}=
\left\{\begin{matrix}
-0.099201891332167, \text{ if } j=1\\
+0.350083208012183, \text{ if } j=2\\
-0.157873787262862, \text{ if }  j=3\\
 +0.440986116874325, \text{ if } j=4
\end{matrix}
\right.
\end{equation}
and
\begin{equation}\label{eq:dlamdRdP}
\dfrac{\partial\lambda_0}{\partial R}=0.007208020694512, \hspace{.5cm}
\dfrac{\partial\lambda_0}{\partial P}=-4.091507558661688,
\end{equation}
which have a natural interpretation, both in terms of the sign and the values. In particular, the fact $\lambda_0$ is increasing in $R$ can be observed in Fig. \ref{mult}, left.

Finally, we show in Fig. \ref{chev} a partial view of the spectrum of the operator $A$ given in \eqref{opA}, with emphasis on non-real eigenvalues. These eigenvalues have been calculated by the Chebyshev Spectral Collocation Method, with two different resolutions, namely $N=30$ and $N=50$ in order to appreciate the values where both coincide and the values where they do not, meaning numerical artifacts. Here $N$ is the maximum degree of the Chebyshev polynomials used on each interval. Observe its resemblance to Fig.\ref{limitspect0}.

\begin{figure}[htpb]
  \centering
  {\includegraphics[scale=0.28]{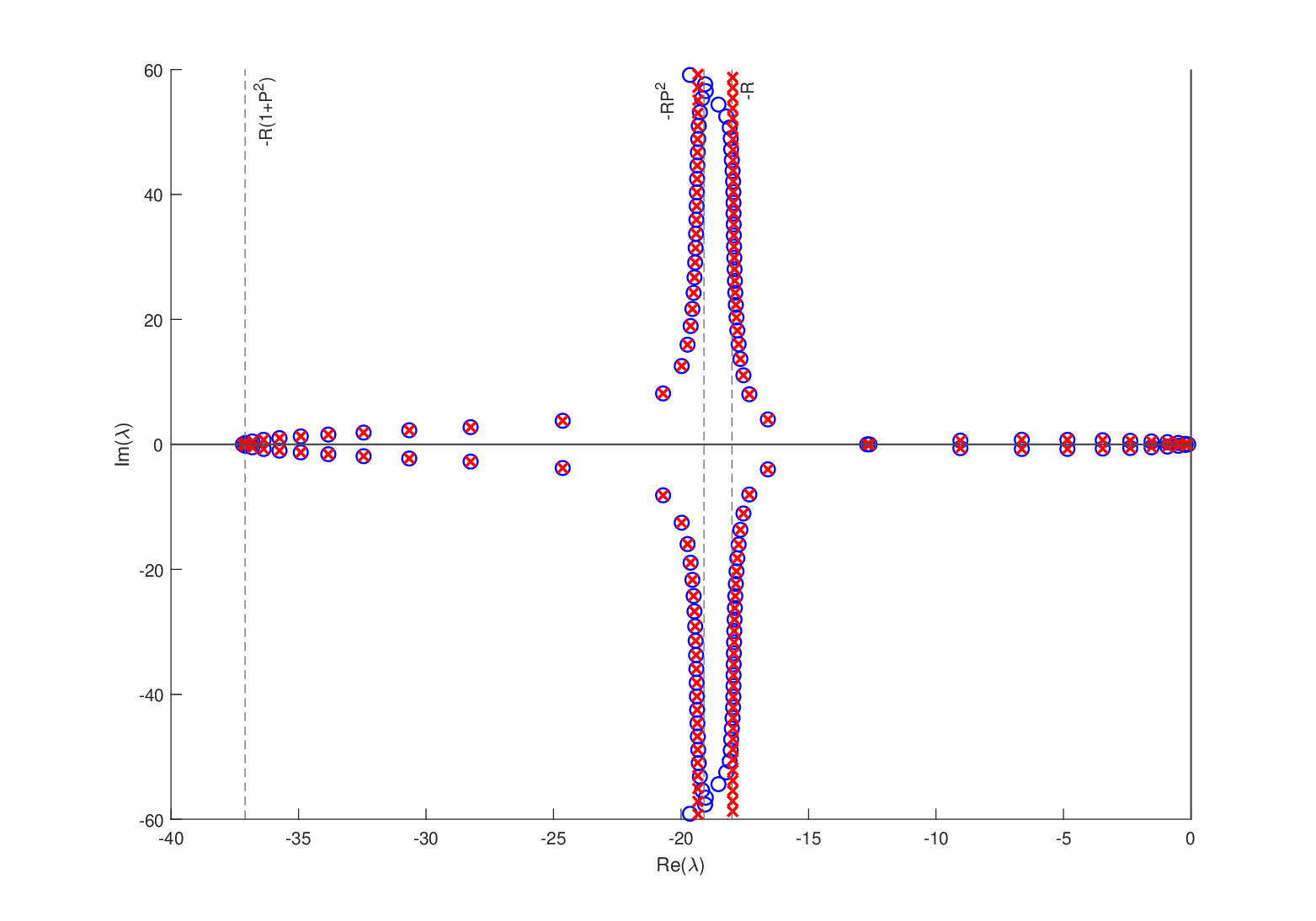}}
    \caption{Partial view of the spectrum of $A$ in the case study. In blue (circles): $N=30$ (low resolution); in red (crosses): $N=45$ (high resolution).}\label{chev}
\end{figure}

\appendix

\section{Appendix: complete calculations of $M_i$, eigenfunctions coefficients and sensitivity analysis}\label{sec:appendix1}
In this section, we present the detailed calculations used in several parts of the previous sections.  For clarity, they are placed in a separate Appendix. The formulas below and procedures turn out to be very useful when implementing numerical calculations.

\subsection{Matrices $M_i$}\label{sec:Mi}
We start by giving the explicit computation of the matrices $M_i(\lambda,x)$ defined in Section \ref{sec:carac}. These explicit formulas are very useful to plot the function $\Delta(\lambda)$ (Fig. \ref{mult}), to compute its asymptotic behaviour for $|\lambda|\to \infty$ (see Section \ref{sec:delta} and Appendix \ref{sec:appendix2}), or to compute its zeroes and, in particular, the dominant eigenvalue $\lambda_0$ (as in Section \ref{case}), for instance.

These matrices can be obtained simply by computing $M_i(\lambda)=M_i(\lambda,x_{i+1})=e^{F_i(\lambda)}$. We have an explicit formula for the matrices $M_i(\lambda)$, which in the cases where $\X{\nu}{1}{i}(\lambda)\neq \X{\nu}{2}{i}(\lambda)$ can be written as:
\begin{equation}\label{eq:Mi}
M_i(\lambda)=\dfrac{1}{\X{\nu}{1}{i}(\lambda)-\X{\nu}{2}{i}(\lambda)}
\begin{pmatrix}
-\X{\phi}{1}{i}(\lambda) e^{\X{\nu}{1}{i}(\lambda)}+\X{\phi}{2}{i}(\lambda) e^{\X{\nu}{2}{i}(\lambda)}&
\dfrac{\X{\phi}{1}{i}(\lambda) \X{\phi}{2}{i}(\lambda)(e^{\X{\nu}{1}{i}(\lambda)}-e^{\X{\nu}{2}{i}(\lambda)})}{RP}\\
(-e^{\X{\nu}{1}{i}(\lambda)}+e^{\X{\nu}{2}{i}(\lambda)})RP&
\X{\phi}{2}{i}(\lambda) e^{\X{\nu}{1}{i}(\lambda)}-\X{\phi}{1}{i}(\lambda) e^{\X{\nu}{2}{i}(\lambda)}
\end{pmatrix},
\end{equation}
where $\X{\nu}{1}{i}(\lambda)$, $\X{\nu}{2}{i}(\lambda)$ are given in \eqref{eq:nu}
and
\begin{equation}\label{eq:phiji}
\X{\phi}{j}{i}(\lambda)=\lambda+R-\X{\nu}{j}{i}(\lambda),\hspace{.5cm} j=1,2.
\end{equation}
Indeed, when $\lambda\in\mathbb{R}$, we have three possible situations:
\begin{enumerate}
\item $\X{\nu}{1}{i}(\lambda)=\overline{\X{\nu}{2}{i}(\lambda)}\in\mathbb{C}\setminus\mathbb{R}$ iff $\lambda\in (\X{\lambda}{1}{i},\X{\lambda}{2}{i})$;
\item $\X{\nu}{1}{i}(\lambda)=\X{\nu}{2}{i}(\lambda)\in\mathbb{R}$ iff $\lambda=\X{\lambda}{1}{i}$ or $\lambda=\X{\lambda}{2}{i}$;
\item $\X{\nu}{1}{i}(\lambda)\neq\X{\nu}{2}{i}(\lambda)$, both real, iff $\lambda\not\in [\X{\lambda}{1}{i},\X{\lambda}{2}{i}]$.
\end{enumerate}
where the $\X{\lambda}{j}{i}$ are the solutions of $\alpha_i(\lambda)^2+ v_i\beta(\lambda)=0$ (we recall that $\alpha_i(\lambda),\beta(\lambda)$ are given in \eqref{alpha} and \eqref{eq:beta}, respectively). That is, the solutions of
$$(v_i +1)^2\lambda^2 + 2\lambda R (v_i +1)(v_i+P^2)+(v_i-P^2)^2R^2 =0. $$
That means that
$$\X{\lambda}{1}{i}=-\dfrac{R(\sqrt{v_i}+P)^2}{v_i+1},\  \ \X{\lambda}{2}{i}=-\dfrac{R(\sqrt{v_i}-P)^2}{v_i+1}.$$
Observe that $\X{\lambda}{1}{i}<0$ and $\X{\lambda}{2}{i}\leq 0$ (actually, $\X{\lambda}{2}{i} =0$ iff  $v_i =P^2$).

Then, developing $\eqref{eq:Mi}$ according to the previous formulas and cases, we can see that the matrices $M_i(\lambda)$ can also be written in the following ways:
\begin{enumerate}
\item If $\X{\nu}{1}{i}(\lambda)=\overline{\X{\nu}{2}{i}(\lambda)}=a_i(\lambda)+b_i(\lambda)i\in\mathbb{C}\setminus\mathbb{R}$, with $a_i(\lambda)=\frac{\alpha_i(\lambda)}{v_i}\in\mathbb{R}$, $b_i(\lambda)=\frac{\sqrt{-(\alpha_i(\lambda))^2-v_i \beta(\lambda)}}{v_i}\in\mathbb{R}$:
$$M_i(\lambda)= \dfrac{e^{a_i(\lambda)}}{b_i(\lambda)} \begin{pmatrix} -\phi_{a_i}(\lambda)\sin{b_i(\lambda)}+b_i(\lambda)\cos{b_i(\lambda)} & \dfrac{RP\sin{b_i(\lambda)}}{v_i} \\
                                        -RP\sin{b_i(\lambda)} & \phi_{a_i}(\lambda)\sin{b_i(\lambda)}+b_i(\lambda)\cos{b_i(\lambda)}
                      \end{pmatrix}$$
                      where $\phi_{a_i}(\lambda)=\lambda+R-a_i(\lambda)$ .
\item If $\nu_i(\lambda):=\X{\nu}{1}{i}(\lambda)=\X{\nu}{2}{i}(\lambda)\in\mathbb{R}$ (that is $\nu_i(\lambda):=a_i(\lambda)=\frac{\alpha_i(\lambda)}{v_i}$):
$$M_i(\lambda) = e^{\nu_i(\lambda)}\begin{pmatrix} 1-\phi_i(\lambda) & \dfrac{(\phi_i(\lambda))^2}{RP} \\
                                           -RP & 1+ \phi_i(\lambda) \end{pmatrix}
$$
where $\phi_i(\lambda)$ is given in \eqref{eq:phiji}.
\item If $\X{\nu}{1}{i}(\lambda)\neq\X{\nu}{2}{i}(\lambda)$, both real. We have $\X{\nu}{1}{i}(\lambda)= a_i(\lambda)+b_i(\lambda)$, $\X{\nu}{2}{i}(\lambda)= a_i(\lambda)-b_i(\lambda)$, with $a_i(\lambda)=\frac{\alpha_i(\lambda)}{v_i}\in\mathbb{R}$, $b_i(\lambda)=\frac{\sqrt{(\alpha_i(\lambda))^2+v_i \beta(\lambda)}}{v_i}\in\mathbb{R}$. Then:
\begin{align}\label{eq:Mi3}
M_i(\lambda) &=  \dfrac{e^{a_i(\lambda)}}{b_i(\lambda)} \begin{pmatrix} -\phi_{a_i}(\lambda)\sinh{b_i(\lambda)}+b_i(\lambda)\cosh{b_i(\lambda)} & \dfrac{RP\sinh{b_i(\lambda)}}{v_i} \\
                                        -RP\sinh{b_i(\lambda)} & \phi_{a_i}(\lambda)\sinh{b_i(\lambda)}+b_i(\lambda)\cosh{b_i(\lambda)}
                      \end{pmatrix}
\end{align}
where $\phi_{a_i}(\lambda)=\lambda+R-a_i(\lambda)$.
\end{enumerate}
Observe that, in all cases, $M_i(\lambda)$ is a real matrix when $\lambda\in\mathbb{R}$.

\subsection{Computation of the constants $\X{C}{j}{i}$ in \eqref{eq:fupA} and $\Xe{C}{j}{i}$ in \eqref{eq:fupAadj} } \label{sec:Cji}

In order to find the explicit values of the coefficients $\X{C}{j}{i}$, $j=1,2$, $i=1,2,3,4$ of the eigenfunction in \eqref{eq:fupA}, we need to impose the eight boundary conditions \eqref{bcn0}. With this, we obtain the following 8-by-8 homogeneous linear system:
\begin{equation}\label{eq:sistCij}
MC=N
\end{equation}
where $C =\left( \X{C}{1}{1},\, \X{C}{2}{1},\, \X{C}{1}{2},\, \X{C}{2}{2},\, \X{C}{1}{3},\, \X{C}{2}{3},\, \X{C}{1}{4},\, \X{C}{2}{4} \right)^T$, $N =\left( 0,0,0,0,0,0,0,0 \right)^T$
and
\tiny
$$M=\begin{pmatrix}
e^{-\X{\nu}{1}{1}} & e^{-\X{\nu}{2}{1}} & -e^{-\X{\nu}{1}{2}} & -e^{-\X{\nu}{2}{2}} & 0 & 0 & 0 & 0 \vspace{0.15cm} \\
0 & 0 & 1 & 1 & -1 & -1 & 0 & 0 \vspace{0.15cm}\\
0 & 0 & 0 & 0 & e^{\X{\nu}{1}{3}} & e^{\X{\nu}{2}{3}} & -e^{\X{\nu}{1}{4}} & -e^{\X{\nu}{2}{4}} \vspace{0.15cm} \\
-e^{-2\X{\nu}{1}{1}} & -e^{-2\X{\nu}{2}{1}} &0 & 0 & 0 & 0 & e^{2\X{\nu}{1}{4}} & e^{2\X{\nu}{2}{4}} \vspace{0.15cm} \\
\X{\phi}{1}{1}e^{-\X{\nu}{1}{1}} & \X{\phi}{2}{1}e^{-\X{\nu}{2}{1}} & -\X{\phi}{1}{2}e^{-\X{\nu}{1}{2}} & -\X{\phi}{2}{2}e^{-\X{\nu}{2}{2}} & 0 & 0 & 0 & 0 \vspace{0.15cm}\\
0 & 0 & 0 & 0 & \X{\phi}{1}{3}e^{\X{\nu}{1}{3}} & \X{\phi}{2}{3}e^{\X{\nu}{2}{3}} & -\X{\phi}{1}{4}e^{\X{\nu}{1}{4}} & -\X{\phi}{2}{4}e^{\X{\nu}{2}{4}} \vspace{0.15cm}\\
v_1\X{\phi}{1}{1}e^{-2\X{\nu}{1}{1}} & v_1\X{\phi}{2}{1}e^{-2\X{\nu}{2}{1}} &0 & 0 & 0 & 0 & -v_4\X{\phi}{1}{4}e^{2\X{\nu}{1}{4}} & -v_4\X{\phi}{2}{4}e^{2\X{\nu}{2}{4}}\vspace{0.15cm} \\
0 & 0 & v_2\X{\phi}{1}{2} & v_2\X{\phi}{2}{2} & -v_3\X{\phi}{1}{3} & -v_3\X{\phi}{2}{3} & 0 & 0
\end{pmatrix}$$
\normalsize
where $\X{\nu}{j}{i}=\X{\nu}{j}{i}(\lambda)$ is defined in \eqref{eq:nu}, and $\X{\phi}{j}{i}=\X{\phi}{j}{i}(\lambda)$ is defined in \eqref{eq:phiji}.

As these are the coefficients of an eigenfunction, this happens to be a consistent undetermined linear system with one degree of freedom. So, stating $\X{C}{1}{1} =1$ (for instance), we can obtain explicit and determined values of the rest of the variables.

In the case study, we have implemented it in \verb"Matlab" to compute the coefficients of the eigenfunction corresponding to the dominant eigenvalue $\lambda_0$, with the following result:
\begin{align*}
\X{C}{1}{1} &= 1  & \X{C}{2}{1} &= 0.026142  \nonumber  \\
\X{C}{1}{2} &= 0.008778 + 0.021136\,i  & \X{C}{2}{2} &= 0.008778 - 0.021136\,i \nonumber  \\
\X{C}{1}{3} &= 1.054202\cdot 10^{-4}   & \X{C}{2}{3} &= 0.017451 \nonumber  \\
\X{C}{1}{4} &= -0.032257 - 0.018820\,i  & \X{C}{2}{4} &= -0.032257 + 0.018820\,i
\end{align*}
The plot of this eigenfunction is given in Fig. \ref{first} (left figure).

Analogously, we can find the explicit values of the coefficients $\Xe{C}{j}{i}$, $j=1,2$, $i=1,2,3,4$ of the eigenfunction of the adjoint problem \eqref{eqnadj}-\eqref{bcnadj}, which, analogously to the obtention of \eqref{eq:fupA}, we can see that has the form
\begin{equation}\label{eq:fupAadj}
\begin{pmatrix} c_i^*(\lambda,x) \\ q_i^*(\lambda,x) \end{pmatrix} = \Xe{C}{1}{i}\begin{pmatrix} \lambda + R - \X{\nu}{1}{i}(\lambda) \\ RP \end{pmatrix} e^{-\X{\nu}{1}{i}(\lambda)x} + \Xe{C}{2}{i}\begin{pmatrix} \lambda + R - \X{\nu}{2}{i}(\lambda) \\ RP \end{pmatrix} e^{-\X{\nu}{2}{i}(\lambda)x}
\end{equation}
for $i = 1, 2, 3, 4$. In this case, we need to impose the eight boundary conditions \eqref{bcnadj}, obtaining an 8-by-8 homogeneous linear system again:
\begin{equation}\label{eq:sistCijadj}
M^*\, C^*=N.
\end{equation}
Now,
$$M^*=$$
\tiny
$$
\begin{pmatrix}
e^{\X{\nu}{1}{1}} & e^{\X{\nu}{2}{1}} & -e^{\X{\nu}{1}{2}} & -e^{\X{\nu}{2}{2}} & 0 & 0 & 0 & 0 \vspace{0.15cm}\\
0 & 0 & 1 & 1 & -1 & -1 & 0 & 0\vspace{0.15cm}\\
0 & 0 & 0 & 0 & e^{-\X{\nu}{1}{3}} & e^{-\X{\nu}{2}{3}} & -e^{-\X{\nu}{1}{4}} & -e^{-\X{\nu}{2}{4}}\vspace{0.15cm} \\
-e^{2\X{\nu}{1}{1}} & -e^{2\X{\nu}{2}{1}} &0 & 0 & 0 & 0 & e^{-2\X{\nu}{1}{4}} & e^{-2\X{\nu}{2}{4}}\vspace{0.15cm}\\
0 & 0 & \X{\phi}{1}{2} & \X{\phi}{2}{2} & -\X{\phi}{1}{3} & -\X{\phi}{2}{3} & 0 & 0\vspace{0.15cm}\\
\X{\phi}{1}{1}e^{2\X{\nu}{1}{1}} & \X{\phi}{2}{1}e^{2\X{\nu}{2}{1}} &0 & 0 & 0 & 0 & -\X{\phi}{1}{4}e^{-2\X{\nu}{1}{4}} & -\X{\phi}{2}{4}e^{-2\X{\nu}{2}{4}}\vspace{0.15cm}\\
v_1\X{\phi}{1}{1}e^{\X{\nu}{1}{1}} & v_1\X{\phi}{2}{1}e^{\X{\nu}{2}{1}} & -v_2\X{\phi}{1}{2}e^{\X{\nu}{1}{2}} & -v_2\X{\phi}{2}{2}e^{\X{\nu}{2}{2}} & 0 & 0 & 0 & 0\vspace{0.15cm}\\
0 & 0 & 0 & 0 & v_3\X{\phi}{1}{3}e^{-\X{\nu}{1}{3}} & v_3\X{\phi}{2}{3}e^{-\X{\nu}{2}{3}} & -v_4\X{\phi}{1}{4}e^{-\X{\nu}{1}{4}} & -v_4\X{\phi}{2}{4}e^{-\X{\nu}{2}{4}}
\end{pmatrix}
$$
\normalsize
with $C^*=\left( \Xe{C}{1}{1},\, \Xe{C}{2}{1},\, \Xe{C}{1}{2},\, \Xe{C}{2}{2},\, \Xe{C}{1}{3},\, \Xe{C}{2}{3},\,\Xe{C}{1}{4},\, \Xe{C}{2}{4} \right)^T$, the same $N=\left( 0,0,0,0,0,0,0,0 \right)^T$, and where $\X{\nu}{j}{i}=\X{\nu}{j}{i}(\lambda)$ and $\X{\phi}{j}{i}=\X{\phi}{j}{i}(\lambda)$, and are defined in \eqref{eq:nu} and \eqref{eq:phiji}, respectively.

As before, this is a consistent undetermined linear system with one degree of freedom. So, we can impose, for instance, $\Xe{C}{1}{1} =1$ and then obtain explicit and determined values of the rest of the variables.

In the case study and for the eigenfunction of the adjoint problem corresponding to the dominant eigenvalue $\lambda_0$, we have the following result:
\begin{align*}
\Xe{C}{1}{1} &= 1, &\Xe{C}{2}{1} &= 2.765583\cdot 10^4 \nonumber \\
\Xe{C}{1}{2} &= (4.431220 - 2.822988\,i)\cdot 10^4, & \Xe{C}{2}{2} &= (4.431220 + 2.822988\,i)\cdot 10^4\nonumber \\
\Xe{C}{1}{3} &= 2.572854\cdot 10^4,  &\Xe{C}{2}{3} &= 6.289587\cdot 10^4\nonumber \\
\Xe{C}{1}{4} &= (-3.273263 - 0.214582\,i)\cdot 10^4,&\Xe{C}{2}{4} &= (-3.273263 + 0.214582\,i)\cdot 10^4
\end{align*}
The plot of this eigenfunction is given in Fig. \ref{first} (right figure).

\subsection{Sensitivity analysis: explicit calculations}\label{sec:derivcasestudy}
In this section, we are going to give the explicit formulas and key points to implement the explicit calculation of the derivatives of an eigenvalue $\lambda$ with respect to the six parameters of the present problem. That is, formulas \eqref{lapri1}, \eqref{lapriR} and \eqref{lapriP}. Recall that $(c,q)$, $(c^*,q^*)$ are the eigenfunction of $A$ and $A^*$, respectively, corresponding to the eigenvalue $\lambda$ and $\overline{\lambda}$. Their formulas are given in \eqref{eq:fupA} and \eqref{eq:fupAadj}, and their coefficients can be computed solving systems \eqref{eq:sistCij} and \eqref{eq:sistCijadj}, (see Section \ref{sec:Cji} for more details). More compactly, we can write these eigenfunctions for $j=1,2$, $i=1,2,3,4$ as:
\begin{align*}
c_i(\lambda,x) &= \X{D}{1}{i}(\lambda) e^{\X{\nu}{1}{i}(\lambda) x} + \X{D}{2}{i}(\lambda) e^{\X{\nu}{2}{i}(\lambda) x} \qquad &c_i^*(\lambda,x) &= \Xe{D}{1}{i}(\lambda) e^{\X{\nu}{1}{i}(\lambda) x} + \Xe{D}{2}{i}(\lambda) e^{\X{\nu}{2}{i}(\lambda) x}\\
q_i(\lambda,x) &= \X{E}{1}{i}(\lambda) e^{\X{\nu}{1}{i}(\lambda) x} + \X{E}{2}{i}(\lambda) e^{\X{\nu}{2}{i}(\lambda) x} \qquad &q_i^*(\lambda,x) &= \Xe{E}{1}{i}(\lambda) e^{\X{\nu}{1}{i}(\lambda) x} + \Xe{E}{2}{i}(\lambda) e^{\X{\nu}{2}{i}(\lambda) x}
\end{align*}
where
\begin{align*}
\X{D}{j}{i}(\lambda) &= \X{C}{j}{i} \left(\lambda+R-\X{\nu}{j}{i}(\lambda)\right) \qquad &\Xe{D}{j}{i}(\lambda) &= \Xe{C}{j}{i} \left(\lambda+R-\X{\nu}{j}{i}(\lambda)\right)\\
\X{E}{j}{i}(\lambda) &= \X{C}{j}{i} RP \qquad &\Xe{E}{j}{i}(\lambda) &= \Xe{C}{j}{i}RP.
\end{align*}
For the rest of the section, we are going to drop the dependence on $\lambda$ to simplify the notation.

The key point to write the derivative formulas in an even more explicit way (easier to implement and calculate) is that everything is explicit, and the integrals that have to be computed are integrals of exponential functions, which can be computed explicitly too. This is what we are going to do in the present section.

We start with $\partial \lambda/\partial v_k$, for a fixed $k\in\{1,2,3,4\}$, whose formula is given in \eqref{lapri1}.
\begin{enumerate}
\item Numerator's first term of \eqref{lapri1}:
\begin{enumerate}
\item If $k=1,3$:
    \begin{align*}
    &-c(x_k^+)\overline{c^*}(x_k^+) = \\
    &-\left( \X{D}{1}{k} \overline{\Xe{D}{1}{k}} e^{\left( \X{\nu}{1}{k}-\overline{\X{\nu}{1}{k}}\right) x_k} + \X{D}{2}{k} \overline{\Xe{D}{2}{k}} e^{\left( \X{\nu}{2}{k}-\overline{\X{\nu}{2}{k}}\right) x_k} +\X{D}{1}{k} \overline{\Xe{D}{2}{k}} e^{\left( \X{\nu}{1}{k}-\overline{\X{\nu}{2}{k}}\right) x_k} +\X{D}{2}{k} \overline{\Xe{D}{1}{k}} e^{\left( \X{\nu}{2}{k}-\overline{\X{\nu}{1}{k}}\right) x_k}  \right)
    \end{align*}
\item If $j=2,4$:
    \begin{align*}
     &c(x_{k+1}^-)\overline{c^*}(x_{k+1}^-) =  \\
     &\X{D}{1}{k} \overline{\Xe{D}{1}{k}} e^{\left( \X{\nu}{1}{k}-\overline{\X{\nu}{1}{k}}\right) x_{k+1}} + \X{D}{2}{k} \overline{\Xe{D}{2}{k}} e^{\left( \X{\nu}{2}{k}-\overline{\X{\nu}{2}{k}}\right) x_{k+1}} +\X{D}{1}{k} \overline{\Xe{D}{2}{k}} e^{\left( \X{\nu}{1}{k}-\overline{\X{\nu}{2}{k}}\right) x_{k+1}} +\X{D}{2}{k} \overline{\Xe{D}{1}{k}} e^{\left( \X{\nu}{2}{k}-\overline{\X{\nu}{1}{k}}\right) x_{k+1}}
     \end{align*}
\end{enumerate}
\item Numerator's second term of \eqref{lapri1}:
\small
    \begin{align*}
    & - \int_{x_k}^{x_{k+1}}c_x(s) \bar{c}^*(s)ds  =\\
    & -\int_{x_k}^{x_{k+1}}\left( \X{D}{1}{k} \overline{\Xe{D}{1}{k}}\X{\nu}{1}{k} e^{\left( \X{\nu}{1}{k}-\overline{\X{\nu}{1}{k}}\right) s} + \X{D}{2}{k} \overline{\Xe{D}{2}{k}}\X{\nu}{2}{k} e^{\left( \X{\nu}{2}{k}-\overline{\X{\nu}{2}{k}}\right) s} +\X{D}{1}{k} \overline{\Xe{D}{2}{k}}\X{\nu}{1}{k} e^{\left( \X{\nu}{1}{k}-\overline{\X{\nu}{2}{k}}\right) s} +\X{D}{2}{k} \overline{\Xe{D}{1}{k}}\X{\nu}{2}{k} e^{\left( \X{\nu}{2}{k}-\overline{\X{\nu}{1}{k}}\right) s}  \right)\,ds
    \end{align*}
\normalsize
\item Denominator of \eqref{lapri1}:
    \begin{align*}
    & \int_{-2}^{2} (c(s) \overline{c^*}(s)+q(s) \overline{q^*}(s)) ds =\\
    & =\sum_{i=1}^4 \int_{x_i}^{x_{i+1}} (c_i(s) \overline{c_i^{*}}(s)+q_i(s) \overline{q_i^{*}}(s)) ds \\
    & = \sum_{i=1}^4 \int_{x_i}^{x_{i+1}}\left( \X{D}{1}{i} \overline{\Xe{D}{1}{i}} e^{\left( \X{\nu}{1}{i}-\overline{\X{\nu}{1}{i}}\right) s} + \X{D}{2}{i} \overline{\Xe{D}{2}{i}} e^{\left( \X{\nu}{2}{i}-\overline{\X{\nu}{2}{i}}\right) s} +\X{D}{1}{i} \overline{\Xe{D}{2}{i}} e^{\left( \X{\nu}{1}{i}-\overline{\X{\nu}{2}{i}}\right) s} +\X{D}{2}{i} \overline{\Xe{D}{1}{i}} e^{\left( \X{\nu}{2}{i}-\overline{\X{\nu}{1}{i}}\right) s}  \right)  \\
    & + \left( \X{E}{1}{i} \overline{\Xe{E}{1}{i}} e^{\left( \X{\nu}{1}{i}-\overline{\X{\nu}{1}{i}}\right) s} + \X{E}{2}{i} \overline{\Xe{E}{2}{i}} e^{\left( \X{\nu}{2}{i}-\overline{\X{\nu}{2}{i}}\right) s} +\X{E}{1}{i} \overline{\Xe{E}{2}{i}} e^{\left( \X{\nu}{1}{i}-\overline{\X{\nu}{2}{i}}\right) s} +\X{E}{2}{i} \overline{\Xe{E}{1}{i}} e^{\left( \X{\nu}{2}{i}-\overline{\X{\nu}{1}{i}}\right) s}  \right)\,ds    \end{align*}
\end{enumerate}
We have not calculated the results of the integrals because, as we will see, it depends on whether the difference of the $\nu$'s in the exponents is zero or not. For instance, in the particular situation of the case study given in Section \ref{case}, for $\lambda=\lambda_0$ (the dominant eigenvalue) we have
\begin{align*}
  &\X{\nu}{1}{1} = 4.940517817243650 + 0.000000000000000i,\  &\X{\nu}{2}{1} &= 0.540070499574931 + 0.000000000000000i\\
  &\X{\nu}{1}{2} = 0.468997654117160 + 1.850683964017590i,\  &\X{\nu}{2}{2} &=  0.468997654117160 - 1.850683964017590i \\
  &\X{\nu}{1}{3} = 3.876353945879901 + 0.000000000000000i,\  &\X{\nu}{2}{3} &=  0.736469716390774 + 0.000000000000000i \\
  &\X{\nu}{1}{4} = -0.361964481599474 + 1.967567941047238i,\ &\X{\nu}{2}{4} &=  -0.361964481599474 - 1.967567941047238i. \\
\end{align*}
This means that
\[
\X{\nu}{1}{k}-\overline{\X{\nu}{1}{k}} =
\begin{cases}
0, & \text{if } k=1,3, \\
2 \Im(\X{\nu}{1}{k}) i,  & \text{if } k=2,4.
\end{cases},\qquad \X{\nu}{1}{k}-\overline{\X{\nu}{2}{k}} =
\begin{cases}
\X{\nu}{1}{k}-\X{\nu}{2}{k} \neq 0, & \text{if } k=1,3, \\
0,  & \text{if } k=2,4.
\end{cases}
\]
So, in the case study, the above formulas for the three terms in $\partial \lambda/\partial v_k$ (for a fixed $k$) turn into:
\begin{enumerate}
\item Numerator's first term:
    \begin{enumerate}
    \item If $k=1,3$:
    $$-c(x_k^+)\overline{c^*}(x_k^+) = -\left( \X{D}{1}{k} \overline{\Xe{D}{1}{k}}  + \X{D}{2}{k} \overline{\Xe{D}{2}{k}}  +\X{D}{1}{k} \overline{\Xe{D}{2}{k}} e^{\left( \X{\nu}{1}{k}-\overline{\X{\nu}{2}{k}}\right) x_k} +\X{D}{2}{k} \overline{\Xe{D}{1}{k}} e^{\left( \X{\nu}{2}{k}-\overline{\X{\nu}{1}{k}}\right) x_k}  \right)$$
    \item If $k=2,4$:
     $$c(x_{k+1}^-)\overline{c^*}(x_{k+1}^-) =  \X{D}{1}{k} \overline{\Xe{D}{1}{k}} e^{ 2\Im(\X{\nu}{1}{k}) i\, x_{k+1}} + \X{D}{2}{k} \overline{\Xe{D}{2}{k}} e^{2\Im(\X{\nu}{2}{k}) i\, x_{k+1}} +\X{D}{1}{k} \overline{\Xe{D}{2}{k}}  +\X{D}{2}{k} \overline{\Xe{D}{1}{k}}   $$
   \end{enumerate}
\item Numerator's second term:
    \begin{enumerate}
    \item If $k=1,3$:
    \begin{align*}
     &- \int_{x_k}^{x_{k+1}}c_x(s) \bar{c}^*(s)ds =
     - \X{D}{1}{k} \overline{\Xe{D}{1}{k}}\X{\nu}{1}{k}  - \X{D}{2}{k} \overline{\Xe{D}{2}{k}}\X{\nu}{2}{k}\\
    & -\dfrac{\X{D}{1}{k} \overline{\Xe{D}{2}{k}}\X{\nu}{1}{k}}{\X{\nu}{1}{k}-\overline{\X{\nu}{2}{k}}} \left(e^{\left( \X{\nu}{1}{k}-\overline{\X{\nu}{2}{k}}\right) x_{k+1}} -e^{\left( \X{\nu}{1}{k}-\overline{\X{\nu}{2}{k}}\right) x_{k}}\right)
    -\dfrac{\X{D}{2}{k} \overline{\Xe{D}{1}{k}}\X{\nu}{2}{k}}{\X{\nu}{2}{k}-\overline{\X{\nu}{1}{k}}} \left(e^{\left( \X{\nu}{2}{k}-\overline{\X{\nu}{1}{k}}\right) x_{k+1}}-e^{\left( \X{\nu}{2}{k}-\overline{\X{\nu}{1}{k}}\right) x_k}\right)
    \end{align*}
    \item If $k=2,4$:
    \begin{align*}
    & - \int_{x_k}^{x_{k+1}}c_x(s) \bar{c}^*(s)ds =
      - \dfrac{\X{D}{1}{k} \overline{\Xe{D}{1}{k}}\X{\nu}{1}{k}}{2\Im(\X{\nu}{1}{k}) i} \left(e^{2\Im(\X{\nu}{1}{k}) i\,x_{k+1}}-e^{2\Im(\X{\nu}{1}{k}) i\,x_{k}}  \right)\\
        &-\dfrac{\X{D}{2}{k} \overline{\Xe{D}{2}{k}}\X{\nu}{2}{k}}{2\Im(\X{\nu}{2}{k}) i}\left( e^{2\Im(\X{\nu}{1}{k}) i\, x_{k+1}}-e^{2\Im(\X{\nu}{1}{k}) i\, x_{k}} \right)
        -\X{D}{1}{k} \overline{\Xe{D}{2}{k}}\X{\nu}{1}{k}  -\X{D}{2}{k} \overline{\Xe{D}{1}{k}}\X{\nu}{2}{k}
    \end{align*}
   \end{enumerate}
\item Denominator. We have the sum of eight terms: we give the explicit formula for each of them, depending on the parity of the subindex.
   \begin{enumerate}
   \item When $i=1,3$:
    \begin{align*}
     \int_{x_i}^{x_{i+1}} (c_i(s) \overline{c_i^{*}}(s)) ds =
     &\  \X{D}{1}{i} \overline{\Xe{D}{1}{i}} + \X{D}{2}{i} \overline{\Xe{D}{2}{i}} +
     \dfrac{\X{D}{1}{i} \overline{\Xe{D}{2}{i}}}{\X{\nu}{1}{i}-\overline{\X{\nu}{2}{i}}}\left( e^{\left( \X{\nu}{1}{i}-\overline{\X{\nu}{2}{i}}\right) x_{i+1}}  -e^{\left( \X{\nu}{1}{i}-\overline{\X{\nu}{2}{i}}\right) x_{i}}  \right) \\
     &+    \dfrac{\X{D}{2}{i} \overline{\Xe{D}{1}{i}}}{\X{\nu}{2}{i}-\overline{\X{\nu}{1}{i}}}\left( e^{\left( \X{\nu}{2}{i}-\overline{\X{\nu}{1}{i}}\right) x_{i+1}}  -e^{\left( \X{\nu}{2}{i}-\overline{\X{\nu}{1}{i}}\right) x_{i}}  \right)
    \end{align*}
    and the same formula for $\int_{x_i}^{x_{i+1}} (q_i(s) \overline{q_i^{*}}(s)) ds$, but with $\X{E}{j}{i}, \Xe{E}{j}{i}$ instead of $\X{D}{j}{i}, \Xe{D}{j}{i}$
   \item When $i=2,4$:
   \begin{align*}
     \int_{x_i}^{x_{i+1}} (c_i(s) \overline{c^{i,*}}(s)) ds =
    &\ \ \dfrac{\X{D}{1}{i} \overline{\Xe{D}{1}{i}}}{2\Im(\X{\nu}{1}{i})i} \left( e^{ 2\Im(\X{\nu}{1}{i})i\, x_{i+1}}  -e^{2\Im(\X{\nu}{1}{i})i\, x_{i}}\right) \\
    &+\dfrac{\X{D}{2}{i} \overline{\Xe{D}{2}{i}}}{2\Im(\X{\nu}{2}{i})i} \left( e^{ 2\Im(\X{\nu}{2}{i})i\, x_{i+1}}  -e^{2\Im(\X{\nu}{2}{i})i\, x_{i}}\right)
    + \X{D}{1}{i} \overline{\Xe{D}{2}{i}} +\X{D}{2}{i} \overline{\Xe{D}{1}{i}}
   \end{align*}
     and the same formula for $\int_{x_i}^{x_{i+1}} (q_i(s) \overline{q_i^{*}}(s)) ds$, but with $\X{E}{j}{i}, \Xe{E}{j}{i}$ instead of $\X{D}{j}{i}, \Xe{D}{j}{i}$
  \end{enumerate}
  \end{enumerate}
Implementing the previous formulas in \verb"Matlab", we can compute the values of the derivatives of $\lambda_0$ (the dominant eigenvalue) in the case study given in formula \eqref{eq:dlamdvi}.
 If we are in other situations other than the case study, the same procedure can be applied and, hence, similar formulas can be derived, depending on the values of $\X{\nu}{j}{i}$ in each interval.

The same ideas can be applied in the rest of the derivatives $\partial \lambda/\partial R$, and $\partial \lambda/\partial P$.

\section{Appendix: asymptotic behaviour of $\Delta(\lambda)$}\label{sec:appendix2}
In this section, we present the computations used in Section~\ref{sec:delta} to determine the asymptotic behaviour of $\Delta(\lambda)$ for $\lambda\in\mathbb{R}$ as $|\lambda|\to\infty$.

In the definition of $\Delta(\lambda)$ (see \eqref{eq:f}) we need both $\det(C(\lambda))$ and ${\rm trace}(C(\lambda))$. Recall that $C(\lambda)$ is defined in \eqref{eq:C} and that the relevant expressions for $M_i(\lambda)$ to compute it when $|\lambda|\to\infty$ correspond to case~3 of Section~\ref{eq:Mi}, namely formula~\eqref{eq:Mi3}. In formula \eqref{eq:detC2}, we have an explicit expression for $\det(C(\lambda))$. However, we do not have an explicit formula for ${\rm trace}(C(\lambda))$, since the trace does not interact in a simple way with matrix products. Therefore, we need to derive an asymptotic expression for ${\rm trace}(C(\lambda))$ to compare it with formula \eqref{eq:detC2}.

Using the definition of each of the previous functions, we are interested in the dominant term (and its sign) of each of them, in terms of $\lambda\in\mathbb{R}$ as $|\lambda|\to\infty$. We are going to use the notation $O(1)$ and $O^+(1)$ given in Remark \ref{rem:Ogran}, that can be generalized to $O(\lambda^k)$ and $O^+(\lambda^k)$.

We start with $a_i(\lambda),b_i(\lambda)$, given in \eqref{eq:Mi3}. Using the asymptotic expansion of $\sqrt{1+x}$ for $x=|\lambda|^{-1}$ small, and taking the dominant terms, we can see that:
$$a_i(\lambda)=\dfrac{v_i-1}{2v_i}\lambda+O(1),\ \ b_i(\lambda)=\dfrac{v_i+1}{2v_i}|\lambda|+\dfrac{R(v_i+P^2)}{2v_i}\textrm{sign}(\lambda) -\dfrac{P^2R^2}{v_i+1}\textrm{sign}(\lambda)\,|\lambda|^{-1}+O\left(|\lambda|^{-2}\right).$$
Now, we use the exponential function properties, the asymptotic expansion of $1/(1+x)$ for $x=|\lambda|^{-1}$ small, and take the dominant terms, and we obtain:
$$\dfrac{e^{a_i(\lambda)}}{b_i(\lambda)}=O^+(1)\exp\left(\dfrac{v_i-1}{2v_i}\lambda\right)\,|\lambda|^{-1}.$$
Again, using exponential properties and taking the dominant terms of each part, we arrive at
$$
\sinh b_i(\lambda)=O^+(1)\left( \exp\left( \dfrac{v_i+1}{2v_i}|\lambda| \right) +o(1)  \right)
, \ \cosh b_i(\lambda) = O^+(1)\left( \exp\left( \dfrac{v_i+1}{2v_i}|\lambda| \right) +o(1)  \right)$$
Also, we have $\phi_{a_i}(\lambda)$ in \eqref{eq:Mi3}, that we can write as
$$\phi_{a_i}(\lambda)=\dfrac{v_i+1}{2v_i}\lambda+O^+(1).$$

Now, we want to compute the dominant terms of $M_i$, which can be written as
 $$M_i(\lambda)=O^+(1)\exp\left(\dfrac{v_i-1}{2v_i}\lambda\right)\left( \exp\left( \dfrac{v_i+1}{2v_i}|\lambda| \right) +o(1)  \right)
 \begin{pmatrix} \mu_{11}^i&\mu_{12}^i\\\mu_{21}^i&\mu_{22}^i\end{pmatrix}|\lambda|^{-1}$$
with
\[
\begin{array}{ccc}
\mu_{11} =
\begin{cases}
-\dfrac{v_i+1}{v_i}\lambda - O^+(1), \ \textrm{ if }\lambda<0\\[0.75em]
-\dfrac{P^2R^2}{v_i+1}\lambda^{-1} + O\left( |\lambda|^{-2}\right), \ \textrm{ if }  \lambda>0
\end{cases}\hspace{-0.3cm},
&
\begin{array}{c}
\mu_{12}=O^+(1)\\[0.7em]
\mu_{21}=-O^+(1)
\end{array},
&
\mu_{22} =
\begin{cases}
\dfrac{P^2R^2}{v_i+1}\lambda^{-1} + O\left( |\lambda^{-2}|\right), \ \textrm{ if } \lambda<0\\[0.75em]
\dfrac{v_i+1}{v_i}\lambda + O^+(1), \ \textrm{ if }  \lambda>0.
\end{cases}
\end{array}
\]
This can be summarised as:
\[
M_i(\lambda) =
O^+(1)\exp\left(\dfrac{v_i-1}{2v_i}\lambda\right)\exp\left(\dfrac{v_i+1}{2v_i}|\lambda| + o(1)\right)|\lambda|^{-1}\,M_i'(\lambda)
\]
where
\[
M_i'(\lambda) =
\begin{cases}
 \qquad\begin{pmatrix} a_i|\lambda|-O^+(1)& O^+(1)\\ -O^+(1) & -b_i|\lambda|^{-1}+ O(|\lambda|^{-2})\end{pmatrix} &\textrm{if }\lambda<0,\\[1.25em]
 \qquad\begin{pmatrix} -c_i |\lambda|^{-1} +O(|\lambda|^{-2})& O^+(1)\\ -O^+(1) & d_i |\lambda| +O^+(1)\end{pmatrix}  &\textrm{if }\lambda>0
\end{cases}
\]
with $a_i,b_i, c_i,d_i>0$ in both cases.
Then one can compute the matrix $C(\lambda)$ as in \eqref{simpleC}, which is the product of the scalar factor
$$\Pi_{i=1}^4 O^+(1)\exp\left(\dfrac{v_i-1}{2v_i}\lambda\right) \exp\left(\dfrac{v_i+1}{2v_i}|\lambda| + o(1)\right) |\lambda|^{-1}$$
times the matrix product
$$
M_1'(\lambda)\cdot
D_{4,1}\cdot
M_2'(\lambda)\cdot M_3'(\lambda)\cdot D_{2,3}\cdot
M_4'(\lambda),$$
where the matrices $D_{i,j}$ are simply
$$
D_{i,j}=\begin{pmatrix} v_i/v_j& 0 \\0&1\end{pmatrix}.
 $$

With the help of a symbolic algebraic manipulator, we obtain the dominant terms of $C(\lambda)$, when $|\lambda|\to\infty$:
$$
 C(\lambda)=
 \begin{cases}
  O^+(1) \exp\left(\left(\sum_{i=1}^4\dfrac{1}{v_i}\right)|\lambda| +o(1)\right)|\lambda|^{-4}
 \begin{pmatrix} (\prod_{i=1}^{4}a_i) \frac{v_2 v_4}{v_1 v_3} |\lambda|^4+ O(|\lambda|^3)& O(|\lambda|^3)\\  O(|\lambda|^3) & O(|\lambda|^2)\end{pmatrix},\ \ &\textrm{ if } \lambda<0\\[2.75em]
 O^+(1)\exp\left(4\lambda +o(1)\right)|\lambda|^{-4}
 \begin{pmatrix} O(|\lambda|^2)& O(|\lambda|^3)\\ O(|\lambda|^3) & (\prod_{i=1}^{4}d_i) \frac{v_2 v_4}{v_1 v_3} |\lambda|^4+O(|\lambda|^3)\end{pmatrix},\ \  &\textrm{ if } \lambda>0
 \end{cases}
$$
with $a_i= d_i=\dfrac{v_i+1}{v_i}$. So,
 \begin{equation}\label{eq:traceinfinity}
 {\rm trace}\left(C(\lambda)\right) = \begin{cases} O^+(1) \exp\left(\left(\sum_{i=1}^4\dfrac{1}{v_i}\right)|\lambda|+o(1)\right)\left(1+O(|\lambda|^{-1})\right) &\textrm{ if } \lambda<0 \\[10pt]
  O^+(1) \exp\left(4\lambda +o(1)\right)\left(1+O(|\lambda|^{-2})\right) &\textrm{ if } \lambda>0\end{cases}
 \end{equation}
This formula is used in Section~\ref{sec:delta} to give the asymptotic behaviour of $\Delta(\lambda)$ for $\lambda\in\mathbb{R}$ as $|\lambda|\to\infty$.

\section{Appendix: the numerical method for the evolution problem}\label{sec:appendix3}
The numerical method used for simulating the countercurrent chromatography, that is, the evolution problem given in \eqref{eqn0t}-\eqref{bcn0t}, is inspired by the one used in \cite{Menacho2011, Menacho2013}. More concretely, it is a sort of volume-of-fluid method combined with the Operator Splitting Method (see W.H. Press et al. \cite{Pressetal}): after each of the four zones has been partitioned in $N_x$ elements of length $\Delta x$, the two simultaneous phenomena involved (advection and interphase mass transfer) are split into two successive calculations, for each time-step $\Delta t$, according to the following scheme.

\begin{enumerate}
\item \textbf{Advection.} First, we calculate the guess-values of the concentration $(\hat{c}, \hat{q})^T$ as the solution of the advection part of problem \eqref{eqn0t}-\eqref{bcn0t}, that is
    \begin{equation*}
    \begin{cases}
        \hat{c}_t+v_i \hat{c}_x=0 \\
        \hat{q}_t-\hat{q}_x=0.
    \end{cases}
    \end{equation*}
    So, we compute $(\hat{c}, \hat{q})^T$ in every element after just a displacement 
    as
    \begin{equation}\label{metnum01}
    \begin{cases}
        \hat{c}_{i,j,k+1}=c_{i,j,k} -v_i p\left(c_{i,j,k}-c_{i,j-1,k}\right) \\
        \hat{q}_{i,j,k+1}=q_{i,j,k} + p\left(q_{i,j+1,k}-q_{i,j,k}\right)
    \end{cases}
    \end{equation}
where the subscripts mean the zone $(i=1,2,3,4)$, the element in each zone $(j=0,1,2 ... N_x)$ and the time-step $(k=0,1,2...N_t)$. The velocity of the solid phase is $v_s=1$, and that of the liquid in the $i$-th zone is $v_i$. The parameter $p$ controls the time-step duration, and must be chosen in the interval
\begin{equation*}
p\in \left(0,\ (\max_i\left\{v_i,v_{s}\right\})^{-1}\right)=\left(0,\ (\max_i\left\{v_i,1\right\})^{-1}\right)
\end{equation*}
implying the corresponding time-step to be
\begin{equation*}
\Delta t=\frac{p}{v_s} \Delta x =p \Delta x,
\end{equation*}
in order to satisfy the Courant-Friedrichs-Lewy condition.
Note that, to apply \eqref{metnum01}, one must take into account the following equivalences, based on the boundary (inter-zone) conditions \eqref{bcn0t}:
\begin{equation*}
\begin{cases}c_{i+1,0,k}:=c_{i,N_x,k},\ \  i=1,3 \\
c_{3,0,k}:=\frac{v_2}{v_3} c_{2,N_x,k}\\
c_{1,0,k}:= \frac{v_4}{v_1} c_{4,N_x,k} \\
q_{i,N_x+1,k}:=q_{i+1,1,k},\ \  i=1,2,3 \\
q_{4,N_x+1,k}:=q_{1,1,k}.
\end{cases}
\end{equation*}
\item \textbf{Interphase mass transfer.} Second, using the guess values $(\hat{c}_{i,j,k+1}, \hat{q}_{i,j,k+1})^T$ obtained in the previous step, we compute the mass transfer in each element to obtain the corrected values $(c_{i,j,k+1}, q_{i,j,k+1})^T$. This calculation only takes into account the mass transfer phenomenon of problem \eqref{eqn0t}-\eqref{bcn0t}, that is:
\begin{equation*}
\begin{cases}
c_t=-RP(Pc-q)\\
q_t=R(Pc-q)
\end{cases}
\end{equation*}
We can solve the previous ODE system explicitly. Taking the guess-value  $(\hat{c}_{i,j,k+1}, \hat{q}_{i,j,k+1})^T$ calculated with \eqref{metnum01} as the initial value, one gets the value for $t=t_{k+1}$ as
\begin{equation*}
\begin{pmatrix}
 c_{i,j,k+1} \\ q_{i,j,k+1}\end{pmatrix} =\frac{1}{1+P^2} \begin{pmatrix}1+P^2 e^{-R(P^2+1)\Delta t} &P (1+ e^{-R(P^2+1)\Delta t}) \\
P(1-e^{-R(P^2+1)\Delta t}) & P^2+e^{-R(P^2+1)\Delta t}  \end{pmatrix}
\begin{pmatrix}\hat{c}_{i,j,k+1} \\ \hat{q}_{i,j,k+1} \end{pmatrix}
\end{equation*}
\end{enumerate}

\section{Acknowledgements}
M. Pellicer and J. Solà-Morales are supported by the Spanish grant PID2021-123903NB-I00 funded by \\ MCIN/AEI/10.13039/501100011033 and by ERDF ``A way of making Europe''.

\end{document}